\documentclass[12pt]{article}
\usepackage{epsfig,psfrag,amsmath,amssymb,latexsym}
\usepackage{amscd}
\usepackage{color}
\usepackage{amsfonts}
\usepackage{graphicx}
\usepackage{wasysym}
\usepackage{mathrsfs}
\usepackage{hyperref}
\numberwithin{equation}{section}
\newtheorem{thm}{Theorem}[section]

\newtheorem{theorem}[thm]{Theorem}
\newtheorem{fact}[thm]{Fact}

\newtheorem{assumption}[thm]{Assumption}
\newtheorem{definition}[thm]{Definition}
\newtheorem{lemma}[thm]{Lemma}
\newtheorem{remark}[thm]{Remark}
\newenvironment{proof}[1][Proof]{\textbf{#1.} }{\ \rule{0.5em}{0.5em}}

\newcommand{\konv}[1]{\stackrel{#1}{\longrightarrow}}

\newcommand{\supp}{\operatorname{supp}}
\newcommand{\hit}{H}

\newcommand{\start}{{\mathrm{s}}}

\newcommand{\new}{\mathrm{new}}

\newcommand{\pstrich}{\mathbb{P}}

\newcommand{\sk}[1]{\left\langle{#1}\right\rangle}
\newcommand{\cL}{\mathcal{L}}

\newcommand{\cH}{\mathcal{H}}

\newcommand{\W}{\mathcal{W}}
\newcommand{\B}{\mathcal{B}}
 
\newcommand{\R}{{\mathbb R}}  
\newcommand{\Q}{{\mathbb Q}}  
\newcommand{\N}{{\mathbb N}}  
\newcommand{\Z}{{\mathbb Z}}  
\newcommand{\F}{{\mathcal{F}}}

\newcommand{\G}{{\mathcal{G}}}  
\renewcommand{\mod}{{\operatorname{mod}}}

\begin{document}
\thispagestyle{empty}

\begin{center}
{\LARGE Random interlacements\\
 for vertex-reinforced jump processes}\\[3mm]
Franz Merkl\footnote{Mathematical Institute, Ludwig-Maximilians-Universit\"at M\"unchen,
Theresienstr.\ 39,
D-80333 Munich,
Germany.
E-mail: merkl@math.lmu.de
}
\hspace{5mm} 
Silke W.W.\ Rolles\footnote{Technische Universit{\"{a}}t M{\"{u}}nchen, 
Zentrum Mathematik, Bereich M5,
D-85747 Garching bei M{\"{u}}nchen,
Germany.
E-mail: srolles@ma.tum.de}
\hspace{5mm} 
Pierre Tarr{\`e}s\footnote{NYU-ECNU 
Institute of Mathematical Sciences at NYU Shanghai, 
Courant Institute of Mathematical Sciences, New York, 
CNRS and Universit\'e Paris-Dauphine, 
PSL Research University, Ceremade,
75016 Paris, France.
E-mail: tarres@nyu.edu}
\\[3mm]
{\small \today}\\[3mm]
\end{center}

\begin{abstract}
We introduce random interlacements for transient vertex-reinforced
jump processes on a general graph $G$. Using 
increasing finite subgraphs $G_n$ of $G$ with wired boundary conditions, 
we show convergence of the vertex-reinforced jump process on 
$G_n$ observed in a finite window to the random interlacement
observed in the same window. \\[2mm]
\end{abstract}

\section{Introduction}

In this paper, we analyze random interlacements for transient 
vertex-reinforced jump processes on infinite graphs. 
This joins two worlds, 
random interlacements for transient Markov processes and vertex-reinforced 
jump processes (VRJP). Random interlacements for transient Markov processes 
are a well studied topic; we describe this theory for 
transient Markovian jump processes in Section~\ref{sec:history-interlacement} 
as an ingredient for the present work. On the other hand, 
vertex-reinforced jump processes starting at a given point 
can be seen as mixture of Markovian jump processes with a mixing measure 
depending on the starting point. However, controlling the behavior of that 
mixing measure as the starting point goes to infinity causes a lot of 
technical problems concerning absolute continuity and uniform integrability. 
These problems become more tractable if the starting point is fixed on a 
finite graph growing towards an infinite graph. However, the random 
jump rates governed by the mixing measure strongly depend on the size
of the finite graph; this makes the question more complicated than 
for classical random walk in a random environment. 
The purpose of this paper is to show that it is 
still possible to obtain a corresponding limiting random interlacement. 
We review the parts of the theory of VRJP that we need in 
Section~\ref{sec:intro-vrjp}. The main result of this paper concerns 
the convergence of loop measures of VRJP on finite pieces of a graph with 
wired boundary conditions to random interlacements as the pieces grow to the 
infinite graph. It is stated in Section~\ref{subsection-approx-of-interlace-by-vrjp}.

\subsection{Vertex-reinforced jump processes}
\label{sec:intro-vrjp}

The vertex-reinforced jump process is a continuous time process $Y=(Y_s)_{s\ge 0}$ 
taking values in the set $V$ of vertices of 
a locally finite connected undirected graph $G=(V,E)$ without direct loops. The edges 
$e=\{x,y\}\in E$ with $x,y\in V$ are assigned conductances $C_e=C_{xy}>0$. 
The process starts in a vertex $o\in V$ and it keeps the memory of the 
local times $L_x(s)$ spent at any vertex $x\in V$ at time $s$, where we 
use the convention that initial local times equal $1=L_x(0)$ for all $x\in V$. 
Given that $Y$ is at vertex $x$ at time $s$, it jumps to a neighboring vertex $y$ 
at rate $C_{xy}L_y(s)$. The process was conceived by Werner and first studied
by Davis and Volkov in \cite{davis-volkov1} and \cite{davis-volkov2}. 
In the present paper, we look at VRJP in a different time scale, called 
\textit{exchangeable time scale}. We encode it as a process $\hat w=(w,l)$
in discrete time decorated with the waiting times $l=(l(k))_{k\in\N_0}$ 
at the vertices $w=(w(k))_{k\in\N_0}$. More precisely, the time change is 
given by 
\begin{align}
t=D(s)=\sum_{x\in V}(L_x(s)^2-1).
\end{align}
We consider the process $Z=(Z_t=Y_{D^{-1}(t)})_{t\ge 0}$; the component 
$w(k)$ means its location immediately before the $k+1$-st jump time,  
and $l(k)$ is the time spent by $Z$ at $w(k)$ between the $k$-th and 
$k+1$-st jump time. 

In the remainder of the article, we fix a vertex $o\in V$ and 
make the following assumption:
\begin{assumption}
VRJP on $G$ starting at $o$ is transient, i.e.\
almost all paths visit every vertex at most finitely often.   
\end{assumption}

In particular, according to corollary 4 in \cite{sabot-tarres2012},
this assumption is fulfilled for $\Z^d$, $d\ge 3$ and large constant 
initial weights $C$. 

In the following, let $\R_+=\{a\in\R:a>0\}$ and $\R_+^0=\{a\in\R:a\ge 0\}$. 
We use the convention $C_{xy}=0$ whenever $\{x,y\}$ is not an edge in $E$. 
Sabot and Tarr\`es \cite{sabot-tarres2012} and 
Sabot and Zeng \cite{sabot-zeng15} showed that the time-changed VRJP 
$Z=(Z_t)_{t\ge 0}$ starting in the vertex $o$ is a Markov jump process
in a random environment. 
Given the starting point $o\in V$, the random environment can be described by 
random variables $\beta=(\beta_x)_{x\in V}$, $\beta_x>0$, having a joint 
law $\rho_o$, introduced in Definition \ref{def:rho-o}, below. 
We realize $\beta$ as canonical process (identity map) on $\R_+^V$. There are random 
variables $u_{o,x}\in\R$, $x\in V$, defined in \eqref{eq:def-u-o-x} below, 
which are functions of $\beta$, 
fulfill the normalization $u_{o,o}=0$, and 
\begin{align}
\label{eq:beta-x-in-terms-of-u}
\beta_x=\frac12\sum_{y\in V}C_{xy}e^{u_{o,y}-u_{o,x}}
\qquad\rho_o\text{-a.s.\ for all }x\in V.
\end{align}
The reason is explained in Remark \ref{rem:beta-in-terms-of-u}, below.
In a fixed environment, the Markov jump process has jump rates 
$\frac12C_{xy}e^{u_{o,y}-u_{o,x}}$ from $x$ to~$y$. Consequently,
$\beta_x$ can be interpreted as the total jump rate away from $x$. 
Although the jump rates are given solely in terms of the variables $u_{o,x}$,
it is still convenient to view the family $\beta$ of total jump rates 
as the basic object because of a
coupling needed in Section~\ref{sec:constr-random-env}. 

Let us describe the Markov jump process in formulas. 
Given a value of the environment~$\beta$, a starting point $z\in V$ (which 
equals $o$ in most cases but not always), and the corresponding 
$u_{o,\cdot}=u_{o,\cdot}(\beta)$, we define a probability law 
$Q_{z,\beta}^G$ on $V^{\N_0}\times\R_+^{\N_0}$ with canonical process
$(w,l)$, encoding a nearest-neighbor
continuous time Markov jump process on $G$ with conductances $C$ 
by the following requirements:
$w(0)=z$ holds 
$Q_{z,\beta}^G$-a.s., and for any $k\in\N_0$, conditionally on
$(w(k'))_{0\le k'\le k}$ and $(l(k'))_{0\le k'<k}$, the joint law of 
$w(k+1)$ and $l(k)$ is characterized by 
\begin{align}
\label{def:Q-o-beta-G}
& Q_{z,\beta}^G(w(k+1)=x,\,l(k)>\ell
\,|\,(w(k'))_{0\le k'\le k},\,(l(k'))_{0\le k'<k})\nonumber\\
=& \frac{C_{xw(k)}e^{u_{o,x}}1_{\{\{x,w(k)\}\in E\}}}{\sum_{y\in V}C_{yw(k)}e^{u_{o,y}}}
\exp\left(-\ell\beta_{w(k)}\right)\nonumber\\
=& \frac{C_{xw(k)}e^{u_{o,x}-u_{o,w(k)}}1_{\{\{x,w(k)\}\in E\}}}{2\beta_{w(k)}}
\exp\left(-\ell\beta_{w(k)}\right). 
\end{align}

Of course, the measures $Q_{z,\beta}^G$, $\rho_o$, and some other objects introduced below
depend also on the choice of the weights $C$. However, this is not displayed
in the notation, as we consider $C$ to be fixed.

\begin{fact}[Variant of Theorem 1 (iii) \cite{sabot-zeng15}]
\label{fact:ST2012-mixture}
Let $P_o$ denote the law of the VRJP $(Z_t)_{t\ge 0}$ in exchangeable 
time scale 
on the infinite graph $G$ encoded as $\hat w=(w,l)$ with starting 
point $o$. There exists a probability measure $\rho_o$ on $\R_+^V$ 
such that for any event $A\subseteq V^{\N_0}\times\R_+^{\N_0}$ one has 
\begin{align}
\label{eq:claim-ST2012-mixture-infinite}
P_o(A)=\int_{\R^V_+} Q_{o,\beta}^G(A)\,\rho_o(d\beta).
\end{align}
More specifically, the probability measure $\rho_o$ on $\R_+^V$ introduced
in Definition \ref{def:rho-o} below fulfills this requirement. 
\end{fact}

The fact that VRJP on infinite graphs is a mixture of Markov jump processes 
was stated in \cite{sabot-zeng15} using the law $\nu_V^C$ on $\beta$ given in 
Section 4 in \cite{sabot-tarres-zeng2017}. However, for our construction 
it is essential to use the law $\rho_o$ defined in Definition 
\ref{def:rho-o} below. We remark that $\rho_o$ is \emph{not} an infinite volume 
version of $\nu_V^C$.

A representation similar to Fact \ref{fact:ST2012-mixture} holds for 
VRJP on finite subgraphs $G_n$ of $G$ with wired boundary conditions with 
the \textit{same} mixing measure $\rho_o$ as is shown in Lemma 
\ref{le:representation-mu-o-n} below. However, the 
laws of the transition probabilities on $G_n$ and on $G$ differ, because 
the transition probabilities on $G_n$ are given in terms of random 
variables $u_{o,x}^{(n)}$ given in formula \eqref{eq:def-u} below, in general 
not equal to $u_{o,x}$.

\paragraph{Comparison of the present approach to the approach in 
\cite{sabot-zeng15}.}
For the following three reasons, the construction from 
\cite{sabot-zeng15} cannot be used directly to provide a consistent 
measure on random interlacements. 
\begin{enumerate}
\item When one uses the infinite volume representation from \cite{sabot-zeng15} 
and then constructs a random interlacement directly given a fixed 
environment, the object thus obtained is not tractable in terms of 
finite volume approximations of VRJP. This makes it difficult to observe 
the properties of that object, in particular its reinforced behavior. 
\item The random environment for the VRJP 
started at the wiring point $\delta_n$ of a finite subgraph $G_n$ of $G$ 
with wired boundary conditions is described by random variables 
$\psi^{(n)}(x)$ introduced in Lemma 2 in \cite{sabot-zeng15}; see also
\eqref{eq:def-psi-n} below. Uniform integrability of
$\psi^{(n)}(x)$, $n\in\N$, is unfortunately unknown, see for instance 
section 2.6 in \cite{sabot-zeng15}. 
Therefore, without a solution of this open 
problem, there is no direct way to start the random interlacement process 
associated to VRJP at infinity. 
\item In \cite{sabot-zeng15}, 
the random environment for the VRJP started at the wiring point 
$\delta_n$ of $G_n$ needs an additional gamma variable 
$\gamma_{\delta_n}$ associated to $\delta_n$; see
formulas (2.4) and (4.2) in \cite{sabot-zeng15}. As $G_n$ 
increases towards $G$, it is unclear how to couple the variables 
$\gamma_{\delta_n}$, $n\in\N$. 
\end{enumerate}
In the present paper, we go around these problems by the following approach.
We start VRJP on a finite graph at the given vertex $o$ rather than 
the wiring points $\delta_n$. We don't use $\gamma_{\delta_n}$, but we 
associate a \textit{single} gamma 
variable $\gamma_o$, not depending on $n$, to $o$ rather than to 
$\delta_n$; more details 
are given in Section \ref{sec:coupling-gamma} below. Our approach 
then involves a Radon-Nikodym derivative modifying the random environment 
measure. It is explained in Section \ref{sec:constr-random-env}.

\subsection{Random interlacements}
\label{sec:history-interlacement}

In order to describe random interlacements associated to VRJP, 
we are interested in Markovian random interlacements in random 
environments. However, to start with, we describe the theory in 
a fixed environment first. It is closely linked to the work 
of Sznitman: Random interlacements were introduced 
for simple random walks in $\Z^d$, $d\ge 3$, by Sznitman in 
\cite{sznitman2010}. In \cite{Teixeira2009},  Teixeira generalized the 
notion of random interlacements to transient random walks on 
weighted graphs. Sznitman \cite{sznitman2012} considered 
random interlacements associated to transient continuous-time jump
processes on weighted graphs. In this paper, we need a variant of 
this construction, including an initial piece of the jump process 
starting at a given point rather than starting at infinity. 
Another difference to the classical theory of random interlacements
is that the law of the transition probabilities for VRJP on 
finite subgraphs of $G$, viewed as a mixture of Markov jump processes,  
depends on the size of the finite subgraph. 
For an introduction to random interlacements see the textbook 
\cite{Drewitz-Rath-Sapozhnikov2014} by Drewitz, R\'ath, and 
Sapozhnikov. 

\paragraph{Introduction of random interlacement with
initial path.}
One ingredient for the present paper are Markovian random interlacements 
in continuous time in a random environment encoded by $\beta$ as above. 
For the moment, let us take $\beta\in\R_+^V$ fixed such that 
$u_{o,\cdot}=u_{o,\cdot}(\beta)$ fulfilling the equality in 
\eqref{eq:beta-x-in-terms-of-u} exists and such that the Markovian
jump process with law $Q_{o,\beta}^G$ is transient. 

In the following ``path'' means nearest-neighbor path. For $I\subseteq\Z$, 
we define the set of paths in $G$ indexed by $I$ which visit every 
vertex at most finitely often: 
\begin{align}
W(I):=&\left\{ \begin{minipage}{11.5cm}
\begin{tabular}{ll}
$(w(k))_{k\in I}\in V^I$: & $\{w(k),w(k+1)\}\in E$ if $\{ k,k+1\}\subseteq I$ \\
& and $|\{k\in I:w(k)=j\}|<\infty$ for all $j\in V$
\end{tabular}
\end{minipage}\right\}.
\end{align}
We introduce the set of paths decorated with waiting times
\begin{align}
\label{eq:hat-W}
\hat W(I):= W(I)\times\R_+^I.
\end{align}
We endow it with its natural $\sigma$-field $\hat\W(I)$.
Typical elements of $\hat W(I)$ are denoted by 
$\hat w=(w,l)=(\hat w(k))_{k\in I}$. We abbreviate
$\hat W^\to:=\hat W(\N_0)$, $\hat W:=\hat W(\Z)$, and use
similar abbreviations
$\hat\W^\to$, $\hat\W$ for the corresponding $\sigma$-fields. 
 
Let $\emptyset\neq K\subseteq V$ be a finite subset. In the spirit of Sznitman
\cite{sznitman2010} \cite{sznitman2012}, we introduce a measure 
$\hat Q_{K,\beta}$ on $(\hat W,\hat\W)$ as follows.
We define the event 
\begin{align}
A_K:=\{ w(0)\in K\text{ and }w(k)\notin K \text{ for all }k>0\}
\subseteq\hat W^\to
\end{align}
that the path $w$ visits $K$ for the last time at index $0$. 
The Markov jump process described by 
$Q^G_{\cdot,\beta}$ is reversible in the sense of Lemma \ref{le:reversibility} in the appendix.  
Motivated by this lemma, we take the unique finite measure $\hat Q_{K,\beta}$ on 
$(\hat W,\hat\W)$ specified by the following 
requirement: For all $x\in V$, $\ell\ge 0$, and $B_1,B_2\in \hat\W(\N)$,
\begin{align} 
\label{eq:def-Q-K}
&\hat Q_{K,\beta}[(\hat w(-n))_{n\in\N}\in B_1,\,
w(0)=x,\, l(0)\ge \ell,\,\hat w|_\N\in B_2]=
\nonumber\\
& \beta_xe^{-\ell\beta_x}e^{2u_{o,x}} 
Q^G_{x,\beta}[\hat w|_\N\in B_1,\, \hat w|_{\N_0}\in A_K]\,
Q^G_{x,\beta}[\hat w|_\N\in B_2]. 
\end{align}
Frequently, we consider elements of $\hat W$ modulo time shifts. 
Therefore we introduce
\begin{align}
\hat W^*:=\hat W/\sim, \text{ where }
\hat w\sim\hat w' \Leftrightarrow \exists m\in\Z\;
\forall k\in\Z: 
\hat w(k)=\hat w'(k+m). 
\end{align}
Let $\pi^*:\hat W\to\hat W^*$ and $\hat\W^*$ respectively denote the 
canonical map and the $\sigma$-field on $\hat W^*$ generated by $\pi^*$. 
We consider the set of equivalence classes of paths which visit a 
finite set $K$: 
\begin{align}
\hat W_K^*=\pi^*[\{(w,l)\in\hat W:\, w(0)\in K\}]. 
\label{eq:def-W-K-star}
\end{align}

The next theorem, proven in the appendix, provides the intensity measure for 
the Poisson 
point process of random interlacements in a fixed environment encoded by 
$\beta$.

\begin{theorem}[Intensity measure]
\label{thm:intensity-measure}
There exists a unique measure $\hat\nu_\beta$ on $(\hat W^*,\hat\W^*)$
such that for any finite $K\subseteq V$, one has 
\begin{align}
\label{eq:vu-on-hat-W-K}
1_{\hat W_K^*}\hat\nu_\beta=\pi^*[\hat Q_{K,\beta}].
\end{align}
It is $\sigma$-finite and it is given by 
\begin{align}
\hat\nu_\beta(A)=\sup_{K\subset V\text{\rm finite}}\pi^*[\hat Q_{K,\beta}](A)
\quad\text{for all }A\in\hat\W^*.
\end{align}
The measure $\hat\nu_\beta$ is not the measure zero: for all finite 
$\emptyset\neq K\subset V$ one has 
\begin{align}
\label{eq:non-triviality-nu-beta}
0<\hat\nu_\beta(\hat W^*_K)=\pi^*[\hat Q_{K,\beta}](\hat W_K^*)<\infty.
\end{align}
\end{theorem}

We define a suitable set of point measures, where the individual points 
consist of pairs $(\hat w,t)$ with a doubly infinite path $\hat w$ and 
a time $t>0$:
\begin{align}
\label{eq:def-omega-leftright}
\Omega^{\leftrightarrow}:=\left\{\begin{minipage}{9.7cm} 
$\omega^\leftrightarrow=\sum_{i\in\N}\delta_{(\hat w_i^*,t_i)}$: 
$\hat w_i^*\in\hat W^*$, $t_i>0$, $t_i\neq t_j$ for $i\neq j$, \\
$\omega^\leftrightarrow(\hat W_K^*\times\R_+^0)=\infty$ 
and $\omega^\leftrightarrow(\hat W_K^*\times[0,t])<\infty$ for all\\
finite $\emptyset\neq K\subset V$ and all $t>0$
\end{minipage}\right\}.
\end{align}
This means that we now have two different time lines: \emph{l-times} 
$l=(l(k))_{k\in\N_0}$ on the one hand and \emph{t-times} $t_i$, $t$ on 
the other hand. They should not be confused with each other. Local times at 
vertices in $V$ are always measured in the l-time line. 
Informally speaking, pairs $(t,l)$ should be compared with 
the lexicographic order, with the t-time being the coarser scale and the 
l-time being the finer scale. 

We endow $\Omega^{\leftrightarrow}$ with the $\sigma$-field generated by 
cylinders. Because of \eqref{eq:non-triviality-nu-beta}, there is a 
Poisson point process with a law $\Q_\beta$, realized as canonical process 
on $\Omega^{\leftrightarrow}$, and having the intensity measure 
\begin{align}
\label{eq:def-intensity-measure}
\hat\nu_\beta(d\hat w^*)\times dt.
\end{align}
It describes random point measures over $\hat W^*\times(0,\infty)$. 
Moreover, we introduce the product measure 
\begin{align}
\label{eq:de-P-o-beta-C}
\Q_{o,\beta}:=Q_{o,\beta}^G\times \Q_\beta \quad\text{ on }\quad 
\Omega:=\hat W^\to\times\Omega^\leftrightarrow.
\end{align}
The measure $\Q_{o,\beta}$ is intended to model 
random interlacements with an initial one-sided infinite path starting 
at $o$ and then infinitely many two-sided infinite paths in a given 
environment encoded by $C$ and $\beta$. 

Using $Q_{o,\beta}^G(\hat W^\to)=1$ from transience, we define a probability 
measure $\pstrich_o$ on $\Omega$ by 
\begin{align} 
\pstrich_o(A):=\int_{\R^V_+}\Q_{o,\beta}(A) \,\rho_o(d\beta)
\end{align}
for any measurable set $A$ with the measure $\rho_o$ describing the 
random environment for VRJP as in Fact \ref{fact:ST2012-mixture}. It models 
random interlacements with an initial piece in a random environment. 

Let $\omega=(\omega_\start,\omega^\leftrightarrow)$ be distributed according to 
$\pstrich_o$ with the given $o\in V$. Here ``$\start$'' stands for \emph{start}. 
Then, note that the initial piece 
$\omega_s$ has the same distribution as the trace together with the 
waiting times of a vertex-reinforced jump process in exchangeable time scale
starting in $o$ with weights $C$.

\subsection{Approximation of random interlacements by VRJP}
\label{subsection-approx-of-interlace-by-vrjp}

VRJP on finite graphs is much better understood than on infinite graphs
because of the explicitly known formulas for the random environment described
in \cite{sabot-tarres2012}. Therefore it is natural to compare the random 
interlacements studied in this paper with 
VRJP on finite subgraphs. For this purpose, we consider a finite observation
window $K\subset V$ and an additional $\delta$ ``at infinity''. 
We study the reductions of the processes 
consisting in an infinite speed up of time whenever the process 
is not in $K$. 

We use the notation $[a,b]$ and $(a,b]$ not only for real intervals 
but also for integer ones.

\paragraph{Finite approximations with wired boundary conditions.}
Let $V_n\uparrow V$ be an increasing
sequence of connected subsets of $V$. We take wired boundary conditions
as follows. Let $\delta$ be a new vertex, not contained in $V$.
Let $G_n=(\tilde{V}_n,\tilde{E}_n)$ be the graph with vertex set
$\tilde{V}_n=V_n\cup\{\delta\}$. There are two types of edges in 
$\tilde{E}_n$: First, all edges $\{x,y\}$ in $E$ with $x,y\in V_n$ 
belong to $\tilde{E}_n$ with inherited conductance
$C^{(n)}_{xy}=C_{xy}$. Second, for any $x\in V_n$ 
with $\{y\in V\setminus V_n:\;\{x,y\}\in E\}\neq\emptyset$
there is an edge $\{x,\delta\}\in\tilde{E}_n$ with conductance
$C_{x\delta}^{(n)}=\sum_{y\in V\setminus V_n}C_{xy}$.
For convenience of notation, we set $C^{(n)}_{xy}=0$ if 
$\{x,y\}\notin\tilde E_n$. Let 
\begin{align}
\label{eq:def-W-n-to}
\hat W^\to_n=&\left\{ \begin{minipage}{12.5cm}
$(w(k),l(k))_{k\in\N_0}\in(\tilde V_n\times\R_+)^{\N_0}$:
$\{w(k),w(k+1)\}\in\tilde E_n$ for all $k\in\N_0$
and $|\{k\in\N_0:w(k)=j\}|=\infty$ for all $j\in\tilde V_n$
\end{minipage}\right\}
\end{align}
denote the set of decorated paths in $G_n$ that visit every vertex 
infinitely often. 

Let $K\subseteq V$ be a finite set with $o\in K$ and set 
$\tilde K:=K\cup\{\delta\}$.

\paragraph{$K^+$-reduction on finite graphs.}
We take $n\in\N$ large enough that $K\subseteq V_n$. Let 
$\hat w=(w,l)\in\hat W^\to_n$. By the definition of $W^\to_n$, one has 
$w(k)\in K$ for infinitely many $k$ and $w(k)=\delta$ for infinitely many $k$. 
Consider the subsequence $(w(k_j),l(k_j))_{j\in\N_0}$
of $\hat w$ consisting only of the pairs $(w(k),l(k))$ with $w(k)\in\tilde K$. 
In this subsequence, finitely many (but not infinitely many) consecutive 
$w(k_j)$ may coincide. We unite these consecutive holding pieces as follows.
Recursively, let 
\begin{align}
\label{eq:def-j-m}
j_0:=0\quad\text{ and }\quad
j_{m+1}:=\min\{j>j_m:w(k_j)\neq w(k_{j-1})\}\text{ for }m\in\N_0. 
\end{align}
The $K^+$-reduction $\hat w^K$ of $\hat w$ is defined as follows:
\begin{align}
\label{eq:def-w-super-K}
& \hat w^K=(w^K(m),l^K(m))_{m\in\N_0}\\
\label{eq:def-l-super-K}
\text{ with }\quad &
w^K(m)=w(k_{j_m})\text{ and }l^K(m)=\sum_{j=j_m}^{j_{m+1}-1} l(k_j)1_{\{w(k_j)\neq\delta\}}.
\end{align}
We emphasize that the local time at $\delta$ is not counted in this 
definition. 

With the name $K^+$-reduction we would like to indicate that we observe the 
process not only in $K$, but a little bit more, namely whenever it is at
$\delta$, but not the local time at $\delta$. This is in contrast to 
the $K$-reduction on the infinite graph introduced in the next paragraph, 
where the process is only observed at $K$.

On the finite graph $G_n$, VRJP is recurrent. Hence, it a.s.\ visits the set 
$K$ infinitely often. On the other hand, we assume VRJP to be transient 
on the infinite graph $G$. Hence, it visits $K$ at most finitely often a.s. 
Extending VRJP on the infinite graph by a vertex-reinforced interlacement 
process this difference disappears, making a direct comparison between the 
two reductions possible.

\paragraph{$K$-reduction on the infinite graph.}
Let $\hat w=(w,l)\in\hat W^\to$. If $w$ does not meet the set $K$ we 
define $\hat w^K$ to be the empty list. Else we proceed as follows. 
By the definition of $W^\to$, one has $w(k)\in K$ for at most finitely many $k$, 
say for $J+1$ time points $k$. Similarly to the above, we 
consider the finite subsequence $(w(k_j),l(k_j))_{j\in[0,J]}$
of $\hat w$ consisting only of the pairs $(w(k),l(k))$ with $w(k)\in K$. 
In this subsequence, some consecutive 
$w(k_j)$ may coincide. We unite them as follows.
Recursively, let $j_0:=0$ and 
$j_{m+1}:=\inf\{j\in(j_m,J]:w(k_j)\neq w(k_{j-1})\}$ for $m\in\N_0$.
Let $M\in\N_0$ be the largest $m$ with $j_m<\infty$.  
The $K$-reduction $\hat w^K$ of $\hat w$ is defined by 
\begin{align}
&\hat w^K=(w^K(m),l^K(m))_{m\in[0,M]}\nonumber\\
\text{ with }\qquad &
w^K(m)=w(k_{j_m})\text{ and }l^K(m)=\sum_{j=j_m}^{(j_{m+1}-1)\wedge J} l(k_j).
\end{align}

\paragraph{$K^+$-reduction for interlacements.}
Recall the definition of $\hat W^*_K$ from \eqref{eq:def-W-K-star}. Let 
\begin{align}
\label{eq:typical-omega1}
\omega=& \left(\omega_\start,\omega^\leftrightarrow
=\sum_{i\in\N}\delta_{(\hat w_i^*,t_i)}\right)
\end{align}
be a typical element of $\Omega$, given in \eqref{eq:de-P-o-beta-C}. 
We consider $(\omega_\start,1_{\hat W^*_K\times\R_+^0}\omega^\leftrightarrow)$;
the second component contains only the loops which hit $K$. We write it as
\begin{align}
1_{\hat W^*_K\times\R_+^0}\omega^\leftrightarrow=\sum_{j\in\N}\delta_{(\hat w_{i_j}^*,t_{i_j})}
\end{align}
with $(i_j)_{j\in\N}$ chosen such that $t_{i_j}$ increases with $j$. 
Given the definition of $\Omega^\leftrightarrow$ as in 
\eqref{eq:def-omega-leftright}, this construction works. 

Let $\hat w_{i_j}=(w_{i_j}(k),l_{i_j}(k))_{k\in\Z}$ be the representative of 
$\hat w_{i_j}^*$ with $w_{i_j}(0)\in K$ and $w_{i_j}(k)\notin K$ for $k<0$.  
The $K^+$-reduction $\omega^K$ of the interlacement $\omega$ is defined 
to be the concatenation of 
\begin{align}
\label{eq:def-omega-K-interlacement}
\omega_\start^K\text{ and all }
(\delta,0), (\hat w_{i_j}|_{\N_0})^K
\text{ with $j$ running through } 1,2,3,\ldots 
\end{align}
In other words, we take the part of the initial piece $\omega_\start$ 
running through $K$ and then infinitely many loops around $\delta$ 
obtained from the $K$-reduction of all $\hat w_{i_j}$, with 
holding times at $\delta$ again not being counted.

Let $P^n_o$ denote the law of the vertex-reinforced jump process 
in exchangeable time scale encoded as 
$\hat w=(w,l)=(\hat w(k))_{k\in\N_0}$ on the finite graph $G_n$ with 
weights $C^{(n)}$ and starting point $o$.

\begin{theorem}[Main result: Convergence of $K^+$-reductions]
\label{thm:convergence-interlacement}
Let $K\subset V$ be finite with $o\in K$. The finite-dimensional 
distributions of the 
$K^+$-reduction of VRJP on $G_n$ converge weakly as $n\to\infty$
to the finite-dimensional distributions of the $K^+$-reduction of the 
random interlacement. More precisely, for all $J\in\N$, it holds 
\begin{align}
\cL_{P^n_o}\left(\hat w^K|_{[0,J]}\right)
\stackrel{\text{w}}{\longrightarrow}
\cL_{\pstrich_o}\left(\omega^K|_{[0,J]}\right)
\quad\text{ as }n\to\infty. 
\end{align}
\end{theorem}

Intuitively speaking, the theorem means the following. Suppose we 
have a finite observation window $K\times[0,J]$, where $K$ refers to location
and $[0,J]$ refers to the observable number of jumps. On the 
one hand, we observe 
the jumping particle of a VRJP on the finite graph $G_n$ whenever
it is inside $K$ or at $\delta$. On the other hand, we observe 
another particle jumping on $K\cup\{\delta\}$
described by the $K^+$-reduction of the random interlacement. 
One may imagine time to run infinitely fast whenever the particle
is not in $K$ including when it is in $\delta$. Then, according to the theorem, 
as $n\to\infty$, in the chosen space-time window, 
we can hardly see any difference between the jumping particle on the 
finite graph and the jumping particle coming from the interlacement 
process. 

\paragraph{Remark.} 
The random environment for VRJP in an appropriate
time scaling has a Bayesian conjugate prior property: Conditioned on 
an initial piece of the path, the future of the path is distributed 
according to a VRJP with updated weights. We expect this property to 
be inherited to random interlacements. Working this out in detail is 
beyond the scope of this paper. 

\paragraph{How this article is organized.}
In Section \ref{sec:constr-random-env}, we construct the measure $\rho_o$
describing the random environment for VRJP. We prove the representation 
of VRJP as a mixture of Markov jump processes on the infinite $G$
stated in Fact \ref{fact:ST2012-mixture} using an analogous representation 
on finite approximating subgraphs $G_n$ of $G$ with the \textit{same} measure 
$\rho_o$; see Lemma \ref{le:representation-mu-o-n}. This construction
uses a martingale discovered by Sabot and Zeng \cite{sabot-zeng15}. 
Section \ref{sec:coupling-gamma} describes the connection between 
the representation of VRJP as a mixture of Markov jump processes 
given by Sabot, Tarr\`es, and Zeng in \cite{sabot-tarres-zeng2017} 
and the measure $\rho_o$.

In Section \ref{sec:proof-main-thm}, we study VRJP and the random interlacement
reduced to a finite observation window. We describe the transition rates 
of these different $K^+$-reductions and prove convergence of the 
rates for the $K^+$-reduction of VRJP on $G_n$ to the 
corresponding rates for the $K^+$-reduction of the random interlacement. 
This yields a proof of our main theorem~\ref{thm:convergence-interlacement}. 

To make the paper more self-contained, we provide a proof of Theorem 
\ref{thm:intensity-measure} in Appendix~\ref{sec:appendix-ppp}.

\section{Construction of random environments for VRJP}
\label{sec:constr-random-env}

\subsection{The random environment associated to a fixed reference vertex}

The mixing measure on the infinite graph $G=(V,E)$ is constructed through
finite volume approximations. Let $G_n$, $n\in\N$, be approximating finite
subgraphs as in Section \ref{subsection-approx-of-interlace-by-vrjp}. 
 
It was shown by Sabot and Tarr\`es in \cite{sabot-tarres2012} that 
the mixing measure for VRJP on $G_n$ can be described in terms of the 
supersymmetric
hyperbolic nonlinear sigma-model $H^{2|2}$ in horospherical coordinates, studied
in \cite{disertori-spencer-zirnbauer2010}.
We define it here through an alternative random Schr\"odinger operator 
construction
given in \cite{sabot-tarres-zeng2017} and \cite{sabot-zeng15}:
There is a probability measure $\rho_\infty$ on $\R_+^V$, depending on the 
graph $G$ and the conductances $C$, with Laplace transform
\begin{align}
\label{eq:laplace-rho-infty}
\int_{\R_+^V}e^{-\sk{\lambda,\beta}}\,\rho_\infty(d\beta)
=&
\exp\left(-\sum_{x,y\in V} C_{xy}\left(\sqrt{(\lambda_x+1)(\lambda_y+1)}-1\right)
\right)\prod_{x\in V}\frac{1}{\sqrt{\lambda_x+1}}
\end{align}
for all $(\lambda_x)_{x\in V}\in(-1,\infty)^V$ having only finitely many nonzero
entries, where we define $\sk{\lambda,\beta}=\sum_{x\in V}\lambda_x\beta_x$;
see proposition 1 in \cite{sabot-tarres-zeng2017} (see also theorem 2.1 in 
\cite{disertori-merkl-rolles2017}) and the Kolmogorov extension theorem
construction used in lemma 1 in \cite{sabot-zeng15}. 
Given $\beta\in\R_+^V$,
let 
\begin{align}
\cH_\beta\in\R^{V\times V},\quad
(\cH_\beta)_{xy}=2\beta_x1_{\{x=y\}}-C_{xy};
\end{align}
recall the convention $C_{xy}=0$ if $\{x,y\}$ is not an edge in $G$.
For any $n\in\N$, given the  finite subset $V_n\subset V$, we introduce 
the restriction $\cH_\beta^{(n)}=((\cH_\beta)_{xy})_{x,y\in V_n}$. Let
\begin{align}
\label{eq:def-B}
B=\{\beta\in\R_+^V:\cH_\beta^{(n)}\in\R^{V_n\times V_n}
\text{ is positive definite for all }n\}.
\end{align}
Note that $\rho_\infty$-a.e.\ $\beta$ belongs to $B$ by definition 1 and 
proposition 1 in \cite{sabot-tarres-zeng2017}. 
For any $\beta$ such that $\cH_\beta^{(n)}$ is positive definite, the vector 
$(\psi^{(n)}(x))_{x\in V_n}$ and its component-wise logarithm 
$(u^{(n)}_x)_{x\in \tilde{V}_n}$
are defined by 
\begin{align}
\label{eq:def-psi-n}
\psi^{(n)}(\delta)=e^{u^{(n)}_\delta}:=1,\quad
(\psi^{(n)}(x))_{x\in V_n}=(e^{u^{(n)}_x})_{x\in V_n}
:=(\cH_\beta^{(n)})^{-1} C_{V_n\delta}^{(n)}
\end{align}
where $C_{V_n\delta}^{(n)}=(C_{x\delta}^{(n)})_{x\in V_n}$;
indeed all entries in $(\cH_\beta^{(n)})^{-1}$ are strictly positive,
as was shown in proposition 2 in \cite{sabot-tarres-zeng2017}, 
which allows us to take the logarithms to define $u^{(n)}$.
If $\cH_\beta^{(n)}$ is not positive definite, we set $u^{(n)}_x=0$
for $x\in\tilde V_n$. We also set 
$u_x^{(n)}=0$ for all $x\in V\setminus V_n$. For $x\in V\cup\{\delta\}$
and the fixed vertex $o\in V$, we define 
\begin{align}
\label{eq:def-u}
u_{o,x}^{(n)}=u_x^{(n)}-u_o^{(n)}. 
\end{align}
In particular, $u_{o,o}^{(n)}=0$. 
Note that for $x\in V_n$ formula \eqref{eq:def-psi-n} implies 
\begin{align}
\label{eq:beta-finite-volume}
\beta_x=\frac{1}{2}\sum_{y\in\tilde V_n}C_{xy}^{(n)}e^{u_y^{(n)}-u_x^{(n)}}
=\frac{1}{2}\sum_{y\in\tilde V_n}C_{xy}^{(n)}e^{u_{o,y}^{(n)}-u_{o,x}^{(n)}}.
\end{align}
For any given $n$, we extend 
$\beta$ to be also defined at $\delta\in\tilde V_n$ by 
\begin{align}
\label{eq:def-beta-delta}
\beta_\delta:=\beta_\delta^{\new,n}
:=\frac12\sum_{y\in\tilde V_n}C_{\delta y}^{(n)} e^{u_y^{(n)}}
=\frac12\sum_{y\in \tilde V_n}C_{\delta y}^{(n)}
e^{u_{o,y}^{(n)}-u_{o,\delta}^{(n)}}.
\end{align}
The dependence of $\beta_\delta$ on $n$ is not displayed in the notation. 
We remark that this quantity is called $\tilde\beta_\delta$ in 
\cite{sabot-tarres-zeng2017}; it does \emph{not} coincide with what is called 
$\beta_\delta$ there. 

Consider a nearest-neighbor continuous-time 
Markov jump process on the finite graph $G_n$ endowed with the weights 
$C^{(n)}$ defined in analogy to \eqref{def:Q-o-beta-G} replacing the 
weighted graph $(V,E,C)$ by $(\tilde V_n,\tilde E_n,C^{(n)})$ and 
$u_{o,\cdot}$ by $u_{o,\cdot}^{(n)}$. For a starting point $z\in\tilde V_n$, the 
corresponding probability law $Q_{z,\beta}^{G_n}$ on 
$\tilde V_n^{\N_0}\times\R_+^{\N_0}$ is defined by the requirements that 
$w(0)=z$ holds $Q_{z,\beta}^{G_n}$-a.s., and for any $k\in\N_0$, conditionally on
$(w(k'))_{0\le k'\le k}$ and $(l(k'))_{0\le k'<k}$, the joint law of 
$w(k+1)$ and $l(k)$ is given by 
\begin{align}
\label{def:Q-o-beta-G-n}
& Q_{z,\beta}^{G_n}(w(k+1)=x,\,l(k)>\ell
\;|\;(w(k'))_{0\le k'\le k},\,(l(k'))_{0\le k'<k})\nonumber\\
=& \frac{C_{xw(k)}^{(n)}e^{u_{o,x}^{(n)}}1_{\{\{x,w(k)\}\in\tilde E_n\}}}{\sum_{y\in\tilde V_n}
C^{(n)}_{yw(k)}e^{u^{(n)}_{o,y}}}
\exp\left(-\ell\beta_{w(k)}\right)\nonumber\\
=& \frac{C_{xw(k)}^{(n)}e^{u^{(n)}_{o,x}-u^{(n)}_{o,w(k)}}1_{\{\{x,w(k)\}\in\tilde E_n\}}}{2\beta_{w(k)}}
\exp\left(-\ell\beta_{w(k)}\right),
\end{align}
where we have used the expressions \eqref{eq:beta-finite-volume} and
\eqref{eq:def-beta-delta} for $\beta$ in the last equation. 

On $\R_+^V$, we define $\F_\infty=\sigma(\beta_x,x\in V)$ and the filtration
\begin{align}
\label{eq:def-filtration-F-n}
\F_n=\sigma(\beta_x,x\in V_n), \quad n\in\N.
\end{align}
By \eqref{eq:def-psi-n}, all $u_x^{(n)}$ are $\F_n$-measurable. 
For any vertex $x\in V_n$, we define a measure $\rho_x^n$ on $(\R_+^V,\F_n)$ by 
\begin{align}
\label{eq:def-nu-o-n}
d\rho_x^n=e^{u_x^{(n)}}\, d\rho_\infty|_{\F_n}
\end{align}

Theorem 3(i) in \cite{sabot-tarres-zeng2017} 
shows that VRJP on $G_n$ starting from $\delta$ is a mixture of the laws 
$Q_{\delta,\beta}^{G_n}$ when $\beta$ is drawn randomly with respect to the 
mixing measure $\rho_\infty$. 
The next lemma provides an analogous result for VRJP on $G_n$ starting 
from $o$ rather than from $\delta$: 

\begin{lemma}
\label{le:mixing-measure-on-finite-graph}
VRJP on $G_n$ starting from $o$ is a mixture of the laws $Q_{o,\beta}^{G_n}$
when $\beta$ is drawn randomly with respect to the mixing measure $\rho^n_o$.  
\end{lemma}
\begin{proof}
Formula (3) in theorem 2 of \cite{sabot-tarres-zeng2017}
shows that the distribution of $u^{(n)}$ with respect to $\rho_\infty$ equals the 
distribution of the supersymmetric hyperbolic nonlinear sigma model
introduced first in formulas (1.2) and (1.5) of 
\cite{disertori-spencer-zirnbauer2010}. Note that in that paper, the 
point $\delta$ is not explicitly mentioned. The pinning strengths
$\varepsilon_x$ of that paper correspond to the weights $C_{x\delta}^{(n)}$. 

The effect of changing the reference point in the $H^{2|2}$-model on $G_n$ 
from $\delta$ to $o$ consists of two steps: First the underlying measure, 
here $\rho_\infty|_{\F_n}$, 
gets an additional Radon-Nikodym-derivative $e^{u^{(n)}_o}$. Second, the 
transformation $u^{(n)}\mapsto u_{o,\cdot}^{(n)}$ given in \eqref{eq:def-u} 
changes the normalization from $u^{(n)}_\delta=0$ to $u^{(n)}_{o,o}=0$; cf.\ 
theorem 2 and section 6 of \cite{sabot-tarres2012}. Using Theorem 3(i) in 
\cite{sabot-tarres-zeng2017} again, this time with starting point $o$ 
rather than $\delta$, the claim follows. 
\end{proof}

The most important case for the vertex $x$ in the following lemma 
is $x=o$. 
\begin{lemma}
\label{le:nu-o-n-consistent}
For any vertex $x\in V$, the collection $(\rho_x^n)_{n\in\N}$ is a consistent family of probability
measures, i.e.\ $\rho_x^{n+1}|_{\F_n}=\rho_x^n$ for all $n\in\N$. 
\end{lemma}
\begin{proof}
By formula (5.26) in \cite{disertori-merkl-rolles2017}, 
\begin{align}
\rho_x^n(\R^V_+)=\int_{\R^V_+}e^{u_x^{(n)}}\, d\rho_\infty=1.
\end{align}
Hence, $\rho_x^n$ is a probability measure. In order to show consistency, 
take an event $A\in\F_n$. We calculate
\begin{align}
\label{eq:int-test-fn-m-n+1}
& \rho_x^{n+1}(A)
=\int_Ae^{u_x^{(n+1)}}\, d\rho_\infty
=\int_AE_{\rho_\infty}\left[\left. e^{u_x^{(n+1)}}\right|\F_n\right]d\rho_\infty.
\end{align}
By Proposition 9 in 
\cite{sabot-zeng15}, $(e^{u_x^{(n)}})_{n\in\N}$ is a martingale with respect 
to $\rho_\infty$ and $(\F_n)_{n\in\N}$; see also Theorem 2.5 of 
\cite{disertori-merkl-rolles2017} 
for a formulation in a notation which is closer to the one used in the 
present paper. This yields $\rho_\infty$-a.s.
\begin{align}
E_{\rho_\infty}\left[\left. e^{u_x^{(n+1)}}\right|\F_n\right]
=e^{u_x^{(n)}}. 
\end{align}
Inserting this in \eqref{eq:int-test-fn-m-n+1} yields the consistency 
as follows:
\begin{align}
\rho_x^{n+1}(A)
=\int_A e^{u_x^{(n)}}\, d\rho_\infty
=\rho_x^n(A).
\end{align}
\end{proof}

\begin{definition}
\label{def:rho-o}
For $x\in V$, let $\rho_x$ denote the unique probability measure on 
$(\R^V_+,\F_\infty)$ with 
restrictions $\rho_x|_{\F_n}=\rho_x^n$ for all $n\in\N$ given by Kolmogorov's 
consistency theorem. 
\end{definition}

For all $o,x\in V$ and $n\in\N$, it follows from \eqref{eq:def-nu-o-n} 
and $u_{o,x}^{(n)}=u_x^{(n)}-u_o^{(n)}$ that 
\begin{align}
\label{eq:radon-nikodym-nu-finite-volume}
\frac{d\rho_x|_{\F_n}}{d\rho_o|_{\F_n}}=
\frac{d\rho_x^n}{d\rho_o^n}=e^{u_{o,x}^{(n)}}.
\end{align}
Recall that $\rho_\infty$ is supported on the set $B$ defined in \eqref{eq:def-B}
so that $\rho_o$ is also supported on the same set $B$.
Indeed, for any fixed $n$, the restriction $\rho_o|_{\F_n}$ is absolutely 
continuous with respect to $\rho_\infty|_{\F_n}$. 

\begin{lemma}
\label{le:mgs-wrt-mu-o}
For all $o,x\in V$, the process $\left(e^{u_{o,x}^{(n)}}\right)_{n\in\N}$ is
a martingale with respect to 
the filtration $(\F_n)_{n\in\N}$ and the measure $\rho_o$. It fulfills 
$E_{\rho_o}[e^{u_{o,x}^{(n)}}]=1$. 
\end{lemma}
\begin{proof}
The claims are consequences of \eqref{eq:radon-nikodym-nu-finite-volume} and 
the fact that $\rho_o$ and $\rho_x$ are probability measures. 
\end{proof}

Being a positive martingale, the process $\left(e^{u_{o,x}^{(n)}}\right)_{n\in\N}$
converges $\rho_o$-almost surely to a limit taking values in $[0,\infty)$. 
We define 
\begin{align}
\label{eq:def-u-o-x}
u_{o,x}:=\lim_{n\to\infty}u_{o,x}^{(n)},
\end{align}
whenever this limit exists in $\R$, and $u_{o,x}:=0$ otherwise. 

\begin{definition}
\label{def:B-prime}
Let $B'$ denote the set of all $\beta\in B$ such that 
$u^{(n)}_{o,x}\to u_{o,x}\in\R$ as $n\to\infty$ for all $x\in V$ 
and the Markov jump process on the infinite 
graph $G$ in the environment $\beta$ with distribution $Q_{o,\beta}^G$ is 
transient. 
\end{definition}

\subsection{Comparison with the approach in \cite{sabot-tarres-zeng2017}}
\label{sec:coupling-gamma}

In this section, we explain the connection between the measure $\rho_o$ 
and the construction of the mixing 
measure used by Sabot and Zeng \cite{sabot-zeng15}, which uses an additional 
gamma random variable. We use this connection to deduce uniform integrability
of $(e^{u_{o,x}^{(n)}})_{n\in\N}$ with respect to $\rho_o$. 

Recall that the random variables $\beta_x$, $x\in V$, denote the canonical
projections on $\R_+^V$, and $o\in V$ is fixed. We enlarge the underlying space 
$\R_+^V$ by an additional 
component, taking $\R_+^V\times\R_+$. The projection to the last coordinate 
is denoted by $\gamma_o$, while the projections to the other components are 
again denoted by $\beta_x$, slightly abusing the notation. We endow 
$\R_+^V\times\R_+$ with the sigma field generated by the projections and with 
the probability measure $\rho_o\times\Gamma$, where $\Gamma$ denotes the 
$\Gamma(\frac12,1)$-distribution. 

Fix $n\in\N$. We define 
\begin{align}
\label{eq:def-beta-new}
\beta_x^\new:=&\beta_x+\delta_{xo}\gamma_o 
\qquad           \text{ for }x\in V.
\end{align}
Recall the definition \eqref{eq:def-beta-delta} of $\beta_\delta^{\new,n}$. 
For $x\in V_n$, we use $\beta_x^{\new,n}$ to be a synonym for $\beta_x^\new$,
and abbreviate $\beta^{\new,n}=(\beta_x^{\new,n})_{x\in\tilde V_n}$.

Let $\nu^{C^{(n)}}_{\tilde V_n}$ denote the measure on $\R_+^{\tilde V_n}$ with 
Laplace transform given by formula \eqref{eq:laplace-rho-infty} with the 
weighted graph $(V,E,C)$ in \eqref{eq:laplace-rho-infty} replaced by 
$(\tilde V_n,\tilde E_n,C^{(n)})$. This measure was introduced in 
\cite{sabot-tarres-zeng2017}.

\begin{lemma}
\label{le:beta-new}
The distribution of $\beta^{\new,n}$ 
with respect to $\rho_o\times\Gamma$ equals the measure $\nu^{C^{(n)}}_{\tilde V_n}$. 
In particular, $\cH_{\beta^{\new,n}}\in\R^{\tilde V_n\times\tilde V_n}$
is $\rho_o\times\Gamma$-a.s.\ positive definite. Moreover, the random vector 
$(u_{o,x}^{(n)})_{x\in\tilde V_n}$ can be  $\rho_o\times\Gamma$-a.s.\ recovered from 
$\beta^{\new,n}$ via $u_{o,o}^{(n)}=0$ and  
\begin{align}
\label{eq:u-in-terms-of-beta-new}
(e^{u_{o,x}^{(n)}})_{x\in\tilde V_n\setminus\{o\}}=((\cH_{\beta^{\new,n}})_{\tilde V_n\setminus\{o\},\tilde V_n\setminus\{o\}})^{-1}C^{(n)}_{\tilde V_n\setminus\{o\},o},
\end{align}
or equivalently, 
\begin{align}
\label{eq:def-u-o1}
&   \beta_x^{\new,n}=\frac12\sum_{y\in \tilde V_n}C_{xy}^{(n)}
e^{u_{o,y}^{(n)}-u_{o,x}^{(n)}},\quad (x\in\tilde V_n\setminus\{o\}) .
\end{align}
\end{lemma}
\begin{proof}
Using the definition \eqref{eq:def-beta-new} of $\beta^{\new,n}$, the 
claim \eqref{eq:def-u-o1} is just a combination of the expression 
\eqref{eq:beta-finite-volume} for $\beta_x$, $x\in V_n$, and the definition
\eqref{eq:def-beta-delta} of $\beta_\delta^{\new,n}$.
 
By Lemma~\ref{le:mixing-measure-on-finite-graph}, $\rho_o^n=\rho_o|_{\F_n}$ 
describes 
the mixing measure for VRJP on $G_n$ starting from $o$, with random transition 
rates expressed in terms of the variables $(u_{o,x}^{(n)})_{x\in\tilde V_n}$; cf.\ 
\eqref{def:Q-o-beta-G-n}. 
Since these variables satisfy the equations 
\eqref{eq:def-u-o1}, Corollary 2 in \cite{sabot-tarres-zeng2017} and the fact $u_{o,o}^{(n)}=0$ imply that $(\beta^{\new,n}_x)_{x\in\tilde V_n}$ 
has distribution $\nu^{C^{(n)}}_{\tilde V_n}$ with respect to $\rho_o\times\Gamma$. 

The measure $\nu^{C^{(n)}}_{\tilde V_n}$ is supported 
on $\{\beta\in\R_+^{\tilde V_n}:\cH_\beta\text{ is positive definite}\}$
by its definition, i.e., Definition 1 in \cite{sabot-tarres-zeng2017}. 
Given invertibility of $(\cH_{\beta^{\new,n}})_{\tilde V_n\setminus\{o\},\tilde V_n\setminus\{o\}}$, 
formula \eqref{eq:u-in-terms-of-beta-new} is just another way of writing 
\eqref{eq:def-u-o1}. 
\end{proof}

We remark that the martingale property of $(e^{u_{o,x}^{(n)}})_{n\in\N}$ stated in 
Lemma \ref{le:mgs-wrt-mu-o} is written with respect to the measure
$\rho_o$ without using $\Gamma$, because $u_{o,x}^{(n)}$ does not depend on 
$\gamma_o$. 

\begin{lemma}
\label{le:uniform-integrability}
For all $o,x\in V$, the sequence $(e^{u_{o,x}^{(n)}})_{n\in\N}$
is uniformly integrable with respect to $\rho_o$. 
\end{lemma}
\begin{proof}
The claimed uniform integrability is essentially contained in Corollary 2 
in \cite{sabot-zeng15}. Indeed, Sabot and Zeng define a family of 
random variables $(u^{(n)}(o,x))_{x\in\tilde V_n}$. As a consequence of 
Lemma \ref{le:beta-new}, its joint law equals 
$\cL_{\rho_o}((u^{(n)}_{o,x})_{x\in\tilde V_n})$. 
Corollary 2 of \cite{sabot-zeng15} implies that for any $x\in V$ the 
sequence $(e^{u^{(n)}(o,x)})_{n\in\N}$ is uniformly integrable, which allows
us to conclude.
\end{proof}

\subsection{The random environment for VRJP on an infinite graph}

\begin{lemma}
\label{le:lim-u-n-finite}
For all $o,x\in V$, the limit of $u_{o,x}^{(n)}$ as $n\to\infty$ exists 
$\rho_o$-almost surely in $\R$. In other words, $u_{o,x}$ is $\rho_o$-almost 
surely given by formula \eqref{eq:def-u-o-x}. 
Moreover, the measures $\rho_x$ and $\rho_o$ are mutually absolutely continuous 
with the Radon-Nikodym derivative
\begin{align}
\label{eq:radon-nikodym-nu-infinite}
\frac{d\rho_x}{d\rho_o}=e^{u_{o,x}}\quad\rho_o\text{-a.s.} 
\end{align}
Furthermore, 
\begin{align}
\label{eq:relation-eu-finite-vol}
e^{u_{o,x}^{(n)}}=E_{\rho_o}[ e^{u_{o,x}}|\F_n] \quad\text{ for all }n
\quad\text{ and }\quad
E_{\rho_o}[ e^{u_{o,x}}]=1.
\end{align}
\end{lemma}
\begin{proof}
Taking the limit as $n\to\infty$ in \eqref{eq:radon-nikodym-nu-finite-volume} 
with the help of the martingale convergence theorem and the uniform 
integrability from Lemma \ref{le:uniform-integrability}, we know that 
\begin{align}
\lim_{n\to\infty}\frac{d\rho_x|_{\F_n}}{d\rho_o|_{\F_n}}=
\lim_{n\to\infty}e^{u_{o,x}^{(n)}}
\end{align}
exists $\rho_o$-almost surely in $\R$ and in $L^1$, and that   
$\lim_{n\to\infty}u_{o,x}^{(n)}\in\R\cup\{-\infty\}$ holds $\rho_o$-almost surely.
We need to exclude $\rho_o$-a.s.\ the value $-\infty$.
Since $\bigcup_{n\in\N}\F_n$ generates $\sigma(\beta_x,x\in V)$, we conclude
$\rho_x\ll\rho_o$ with the Radon-Nikodym-derivation 
$d\rho_x/d\rho_o=\lim_{n\to\infty}e^{u_{o,x}^{(n)}}$. Interchanging $x$ and $o$ it 
follows also that $\rho_o\ll\rho_x$ and $d\rho_x/d\rho_o>0$ holds $\rho_o$-a.s.
Hence, $\rho_o$-a.s., $\lim_{n\to\infty}u_{o,x}^{(n)}>-\infty$. This shows 
that indeed $u_{o,x}$ is given by formula \eqref{eq:def-u-o-x} $\rho_o$-a.s.\
and that the claim \eqref{eq:radon-nikodym-nu-infinite} is valid. 

We conclude that $(e^{u_{o,x}^{(n)}})_{n\in\N}$ is a uniformly integrable 
martingale converging to $e^{u_{o,x}}$ in $L^1(\rho_o)$ and $\rho_o$-a.s.
The first equation in \eqref{eq:relation-eu-finite-vol} follows.
$L^1$-convergence and Lemma \ref{le:mgs-wrt-mu-o} imply 
the last equation in \eqref{eq:relation-eu-finite-vol}.
\end{proof} 

\begin{remark}
\label{rem:beta-in-terms-of-u}
We remark that formula \eqref{eq:beta-x-in-terms-of-u} is 
a consequence of \eqref{eq:beta-finite-volume} and the $\rho_o$-almost
sure convergence of $u_{o,x}^{(n)}$, $x\in V$, as $n\to\infty$ to $u_{o,x}$
stated in Lemma \ref{le:lim-u-n-finite}.  In particular, one has 
$\rho_o(B')=1$ and the equation in
\eqref{eq:beta-x-in-terms-of-u} holds for all $\beta\in B'$ with 
$B'$ given in Definition~\ref{def:B-prime}. 
\end{remark}

The VRJP in exchangeable time scale on the finite graph $G_n$ can not only 
be described as a mixture of Markov jump processes with respect to $\rho_o^n$,
but also with respect to its extension $\rho_o$:

\begin{lemma}
\label{le:representation-mu-o-n}
For any event 
$A\subseteq\tilde V_n^{\N_0}\times\R_+^{\N_0}$, one has 
\begin{align}
\label{eq:claim-ST2012-mixture-finite-volume}
P_o^n(A)=\int_{\R_+^V} Q_{o,\beta}^{G_n}(A)\,\rho_o(d\beta).
\end{align}
\end{lemma}
\begin{proof}
By Lemma~\ref{le:mixing-measure-on-finite-graph}, the claim 
holds with $\rho_o$ replaced by $\rho_o^n$. Since 
$Q_{o,\beta}^{G_n}(A)$ is $\F_n$-measurable and $\rho_o^n=\rho_o|_{\F_n}$, 
the claim follows. 
\end{proof}

\medskip\noindent\begin{proof}[Proof of Fact \ref{fact:ST2012-mixture}]
It suffices to show 
\begin{align}
E_{P_o^n}[F]
\konv{n\to\infty}\int_{\R_+^V}E_{Q_{o,\beta}^G}[F]\,\rho_o(d\beta)
\end{align}
for functions $F(\hat w):=f(\hat w|_{[0,J]})$ with any $J\in\N$ and any 
bounded measurable function $f:\hat W([0,J])\to\R$. 
Let $\Pi_J$ denote the set of all paths 
$\pi=(\pi_0,\pi_1,\ldots,\pi_{J+1})\in V^{[0,J+1]}$ in $G$ which start at $o$. 
Clearly, $\Pi_J$ is a finite set. Take $N$ large enough that 
any path in $\Pi_J$ does not leave $V_N$. Let $\beta\in B'$. Because 
$\{u^{(n)}_{o,x}(\beta):x\in V_N, n\ge N\}$ for the given $\beta$ is bounded, 
dominated convergence yields for $n\ge N$
\begin{align}
E_{Q_{o,\beta}^{G_n}}[F]
= & \sum_{\pi\in\Pi_J}\int\limits_{\R_+^{[0,J]}} f(\pi|_{[0,J]},\ell)
\prod_{j=0}^J \frac{C_{\pi_j\pi_{j+1}}}{2}e^{u^{(n)}_{o,\pi_{j+1}}-u^{(n)}_{o,\pi_j}}
e^{-\beta_{\pi_j}\ell(j)}
\, d\ell \konv{n\to\infty}E_{Q_{o,\beta}^G}[F].
\end{align}
Hence, using Lemma \ref{le:representation-mu-o-n} and dominated convergence
again, we obtain 
\begin{align}
E_{P_o^n}[F]=\int_{\R_+^V}E_{Q_{o,\beta}^{G_n}}[F]\,\rho_o(d\beta)
\konv{n\to\infty}\int_{\R_+^V}E_{Q_{o,\beta}^G}[F]\,\rho_o(d\beta).
\end{align}
\end{proof}

\section{Proof of the main result}
\label{sec:proof-main-thm}
\subsection{Transition rates of various reductions}

For $\hat w=(w,l)\in\hat W^\to$ or $\hat w\in\hat W^\to_n$ for some $n$, 
the hitting time and the return time of a set $A$ are defined by 
\begin{align}
\label{def-H-A}
\hit_A(\hat w)= & \inf\{k\ge 0 :\; w(k)\in A\},\\
\widetilde\hit_A(\hat w)=& \inf\{k\ge 1 :\; w(k)\in A\},
\end{align}
respectively. 
If $A=\{y\}$ is a singleton, we write $\hit_y=\hit_{\{y\}}$ and 
$\widetilde\hit_y=\widetilde\hit_{\{y\}}$. 

Let $K\subset V$ be a finite set with $o\in K$. Consider $n$ large 
enough so that $K\subseteq V_n$. We define for $x,y\in K$ 
\begin{align}
\label{eq:def-e-K-n}
e_K^n(x) = & e_{K,\beta}^n(x) 
:=1_{\{x\in K\}}e^{2u_{o,x}^{(n)}}Q_{x,\beta}^{G_n}(\hit_\delta<\widetilde\hit_K),\\
q_K^n(x,y) = & q_{K,\beta}^n(x,y) 
:=1_{\{x\in K\}}Q_{x,\beta}^{G_n}(1<\widetilde\hit_K=\widetilde\hit_y<\hit_\delta).
\end{align}
Note that $\{\hit_\delta<\widetilde\hit_K\}$ means the event to 
exit $K$ immediately and reach $\delta$ before returning to $K$ and
the event $\{1<\widetilde\hit_K=\widetilde\hit_y<\hit_\delta\}$ means that the 
walk exits $K$ immediately and reenters it at $y$ before hitting $\delta$. 
The corresponding quantities in infinite volume are given by
\begin{align}
\label{eq:def-e-K-u-infinite-volume}
e_K(x) = & e_{K,\beta}(x) :=
1_{\{x\in K\}}e^{2u_{o,x}}Q_{x,\beta}^G(\widetilde\hit_K=\infty), \\
q_K(x,y) = & q_{K,\beta}(x,y) := 
1_{\{x\in K\}}Q_{x,\beta}^G(1<\widetilde\hit_K=\widetilde\hit_y<\infty).
\end{align}
Similarly to the above, the event $\{\widetilde\hit_K=\infty\}$ means 
that the walk exits $K$ immediately and never returns to it, and 
the event $\{1<\widetilde\hit_K=\widetilde\hit_y<\infty\}$ means that 
the walk exits $K$ immediately and reenters it at $y$. 

Recall that for any fixed $n\in\N$, the expression $\beta_\delta$ is a 
synonym for $\beta_\delta^{\new,n}$, which does not display the dependence
on $n$. 

\begin{lemma}
For all $\beta\in B$, all finite $\emptyset\neq K\subset V$, $x\in K$, 
and all $n\in\N$, one has 
\begin{align}
\label{eq:const-intensity}
\beta_\delta e^{2u_{o,\delta}^{(n)}}
Q_{\delta,\beta}^{G_n}[\text{first excursion hits }K \text{ first in }x]
=& \beta_x e_K^n(x).
\end{align}
Summing over $x\in K$, we have 
\begin{align}
\beta_\delta e^{2u_{o,\delta}^{(n)}}
Q_{\delta,\beta}^{G_n}[\text{first excursion hits }K]
= & \sum_{x\in K}\beta_x e_K^n(x).
\label{eq:const-intensity-summed}
\end{align}
\end{lemma}
\begin{proof}
We calculate 
\begin{align}
Q_{\delta,\beta}^{G_n}[\text{first excursion hits }K\text{ first in }x]
=&\sum_{\pi\in\Pi_x}Q_{\delta,\beta}^{G_n}(\underline{\pi}),
\end{align}
where we sum over the set $\Pi_x$ of all finite paths $\pi$ from 
$\delta$ to $x$ hitting $K$ for the first time in $x$ and 
visiting $\delta$ only at the start; the event that 
the process initially follows $\pi$ is denoted by $\underline\pi$. 
For 
$\pi=(\pi_0,\pi_1,\ldots,\pi_m)\in\Pi_x$, one has by the reversibility 
\eqref{eq:reversibility-single-path} from the appendix 
\begin{align}
\beta_\delta e^{2u_{o,\delta}^{(n)}} Q_{\delta,\beta}^{G_n}(\underline\pi)
=\beta_x e^{2u_{o,x}^{(n)}}
Q_{x,\beta}^{G_n}(\underline{\pi^\leftrightarrow}),
\end{align}
where $\pi^\leftrightarrow=(\pi_m,\pi_{m-1},\ldots,\pi_0)$ denotes the reversed
path. Consequently, we conclude
\begin{align}
& \beta_\delta e^{2u_{o,\delta}^{(n)}} Q_{\delta,\beta}^{G_n}[\text{first excursion hits }K\text{ first in }x]\nonumber\\
=&\beta_xe^{2u_{o,x}^{(n)}}\sum_{\pi\in\Pi_x}Q_{x,\beta}^{G_n}(\underline{\pi^\leftrightarrow})
=\beta_x e_K^n(x);
\end{align}
in the last step we replaced the sum over $Q_{x,\beta}^{G_n}(\underline{\pi^\leftrightarrow})$
by the sum of $Q_{x,\beta}^{G_n}(\underline{\pi})$, where $\pi$ runs 
over all paths from $x$ to $\delta$ which hit $\delta$ only at the end and 
reach $\delta$ before returning 
to $K$. 
\end{proof}

\begin{lemma} {\bf(Modified $K^+$-reduction of Markov jump processes - finite volume)}
\label{le:transition-rates-K-reduction-finite-volume}
Let $n\in\N$ and consider a given $\beta\in B$. 
Let $K\subset V_n$ with $o\in K$. 
We define a modified $K^+$-reduction 
$\hat w_\mod^K=(w^K(m),l_\mod^K(m))_{m\in\N_0}$ of the 
Markov jump process $\hat w$ on the finite graph $G_n$ in the 
environment $\beta$, described by the probability measure $Q^{G_n}_{o,\beta}$
just as the $K^+$-reduction in formulas 
\eqref{eq:def-j-m}--\eqref{eq:def-l-super-K}
except that the local time in $\delta$, which was ignored in 
\eqref{eq:def-l-super-K}, is now counted, but rescaled:
\begin{align}
\label{eq:def-l-mod-rescaled}
l_\mod^K(m)=\sum_{j=j_m}^{j_{m+1}-1} l(k_j)
\left(e^{-2u_{o,\delta}^{(n)}}1_{\{w(k_j)=\delta\}}+1_{\{w(k_j)\neq\delta\}}\right). 
\end{align}
Then, $\hat w_\mod^K$ is a Markov jump process on $\tilde K$ with 
respect to $Q^{G_n}_{o,\beta}$. 
Its rates for transitions $x\to y$ with different $x,y\in\tilde K$ 
are given by
\begin{align}
& Q^{G_n}_{o,\beta}[w^K(k+1)=y,l_\mod^K(k)<\ell+d\ell\,|\, 
w^K(k)=x,l_\mod^K(k)\ge\ell]
\nonumber\\
=&\left\{\begin{array}{ll}
\left(\frac12 C_{xy}e^{u_{o,y}^{(n)}-u_{o,x}^{(n)}}+\beta_xq_{K,\beta}^n(x,y)\right)
d\ell+o(d\ell) & \;\text{ for }x,y\in K,\\
\beta_xe^{-2u_{o,x}^{(n)}}e_{K,\beta}^n(x)d\ell+o(d\ell) & \;\text{ for }x\in K, y=\delta, \\
\beta_y e_{K,\beta}^n(y)d\ell+o(d\ell) & \;\text{ for }x=\delta, y\in K
\end{array}\right.
\label{eq:transition-rates-finite-volume}
\end{align}
as $d\ell\downarrow 0$. 
\end{lemma}
\begin{proof}
The jumps from $x\in K$ to $y\in K$ originate from two sources. Either 
the original walk jumps along an edge directly from $x$ to $y$, which it 
does at rate $\frac12 C_{xy}e^{u_{o,y}^{(n)}-u_{o,x}^{(n)}}$, or it leaves $K$ at 
$x$ and reenters at $y$. 
Conditionally on jumping away from $x$, which occurs at rate $\beta_x$, 
the random walker leaves $K$ and reenters $K$ at $y$ before hitting 
$\delta$ with probability $q_{K,\beta}^n(x,y)$. This explains the second 
summand in the first line on the right-hand side of 
\eqref{eq:transition-rates-finite-volume}. The argument for transitions 
$K\ni x\to\delta$ is similar: The vertex $x\in K$ is left at rate 
$\beta_x$, and conditionally on leaving it, the probability to 
exit $K$ immediately and hitting $\delta$ before reentering $K$ equals 
$e^{-2u_{o,x}^{(n)}}e_{K,\beta}^n(x)$; the factor
$e^{-2u_{o,x}^{(n)}}$ removes the normalization $e^{2u_{o,x}^{(n)}}$ in the definition 
\eqref{eq:def-e-K-n} of $e_{K,\beta}^n(x)$. Finally, the rate of the 
original walk to leave $\delta$, without rescaling local times at $\delta$,
equals $\beta_\delta$. The rescaling with the factor $e^{-2u_{o,\delta}^{(n)}}$
yields the modified rate $\beta_\delta e^{2u_{o,\delta}^{(n)}}$.
Multiplying 
it with the probability that the first excursion from $\delta$ hits 
$K$ first in $y$, formula \eqref{eq:const-intensity} yields the 
rate $\beta_y e^n_{K,\beta}(y)$ for transitions from $\delta$ to $y$. 
\end{proof}

\begin{lemma}[$K$-reduction of Markov jump processes -- infinite volume]
\label{le:rates-reduction-path}\mbox{}\\
Consider $\beta\in B'$, cf.\ Definition \ref{def:B-prime}.  
Take a finite subset $K\subset V$ with $o\in K$. Consider a
Markov jump process with absorption having state space $K\cup\{\bot\}$, 
where $\bot$ means absorption, with the following jump rates 
\begin{align}
\label{eq:transition-rates-short1}
& \frac12 C_{xy}e^{u_{o,y}-u_{o,x}}+\beta_xq_{K,\beta}(x,y)
\quad\text{ for transitions }x\to y\text{ with } x,y\in K,\\
& \beta_xe^{-2u_{o,x}}e_{K,\beta}(x)\quad\text{ for transitions }
x\to\bot\text{ with }x\in K.
\label{eq:transition-rates-short2}
\end{align}
The law $Q_{z,\beta}^{K^+}$ of this Markov jump process started in any $z\in K$
and stopped 
immediately before being absorbed equals the law of the $K$-reduction 
$\hat w^K=(w^K,l^K)$ with respect to $Q_{z,\beta}^G$.
\end{lemma}
\begin{proof}
The proof is almost the same as the proof of Lemma 
\ref{le:transition-rates-K-reduction-finite-volume}. 
The jumps from $x$ to $y$ originate from two sources. Either the original walk
jumps along an edge directly from $x$ to $y$, which it does at rate
$\frac12 C_{xy}e^{u_{o,y}-u_{o,x}}$, or it leaves $K$ at $x$ and reenters at $y$. 
Conditionally on jumping away from $x$, which occurs at rate $\beta_x$, 
the random walker leaves $K$ and reenters $K$ at $y$ with probability
$q_{K,\beta}(x,y)$. This explains the second summand in 
\eqref{eq:transition-rates-short1}. 
The argument for \eqref{eq:transition-rates-short2} is similar. The factor
$e^{-2u_{o,x}}$ removes the normalization $e^{2u_{o,x}}$ in the definition 
\eqref{eq:def-e-K-u-infinite-volume} of $e_{K,\beta}(x)$. 
\end{proof}

\medskip
In order to phrase a slightly stronger version of the main 
Theorem \ref{thm:convergence-interlacement}, 
we define also a modified $K^+$-reduction for interlacements, which 
gives rise to the following transition probabilities described in 
Lemma \ref{le:transition-rates-K-reduction}.

\begin{lemma}[Modified $K^+$-reduction of interlacements]
\label{le:transition-rates-K-reduction}
Let $\beta\in B'$ and let 
$K\subset V$ be finite with $o\in K$. Define
a modified $K^+$-reduction $\omega^K_\mod=(w^K(k),l^K_\mod(k))_{k\in\N_0}$ 
for interlacements
$\omega$ as in \eqref{eq:typical-omega1} by the following modified version of 
\eqref{eq:def-omega-K-interlacement}, where the local time at $\delta$ is now 
counted and equals the increment of the $t$-parameter of the interlacement: 
$\omega_\mod^K$ is defined  
to be the concatenation of $\omega_\start^K$ and all 
$(\delta,t_{i_j}-t_{i_{j-1}})$, $(\hat w_{i_j}|_{\N_0})^K$ with $j$ running through 
$1,2,3,\ldots$, where we use the convention $t_{i_0}:=0$. 

The modified $K^+$-reduction $\omega^K_\mod$ is a Markov jump process 
on $\tilde K$ with respect to the law 
$\Q_{o,\beta}$ of the interlacement process in the environment $\beta$.
Its rates for transitions $x\to y$ with different 
$x,y\in\tilde K=K\cup\{\delta\}$ are given by
\begin{align}
& \Q_{o,\beta}[w^K(k+1)=y,l^K_\mod(k)<\ell+d\ell|\, w^K(k)=x,l^K_\mod(k)\ge\ell]
\nonumber\\
=&\left\{\begin{array}{ll}
\left(\frac12 C_{xy}e^{u_{o,y}-u_{o,x}}+\beta_xq_{K,\beta}(x,y)\right)
d\ell+o(d\ell) & \;\text{ for }x,y\in K,\\
\beta_xe^{-2u_{o,x}}e_{K,\beta}(x)d\ell+o(d\ell) & \;\text{ for }x\in K, y=\delta, \\
\beta_ye_{K,\beta}(y)d\ell+o(d\ell) & \;\text{ for }x=\delta, y\in K
\end{array}\right.
\label{eq:transition-rates}
\end{align}
as $d\ell\downarrow 0$. 
\end{lemma}

\paragraph{Remark.}
If we do not rescale the time in \eqref{eq:def-l-mod-rescaled}, the rate 
to jump from $\delta$ to $y$ in the last line of 
\eqref{eq:transition-rates-finite-volume} gets an additional factor 
$e^{-2u_{o,\delta}^{(n)}}$, which 
has no counterpart in the infinite volume version \eqref{eq:transition-rates}. 
We do not know almost sure convergence of this factor $e^{-2u_{o,\delta}^{(n)}}$ 
with respect to $\rho_o$ as it is an open question whether this measure is 
absolutely continuous
with respect to $\rho_\infty$. For this reason, 
we have ignored the local time at $\delta$ in the $K^+$-reduction. 

Under the assumption that $(e^{u_o^{(n)}})_{n\in\N}$ is uniformly integrable
with respect to $\rho_\infty$, which is unknown to hold, the measure $\rho_o$ is 
absolutely continuous with respect to $\rho_\infty$. In that case, 
we have a $\rho_o$-a.s.\ limit $e^{2u_{o,\delta}}$ of $(e^{2u_{o,\delta}^{(n)}})_{n\in\N}$. 
Changing then the intensity measure described in 
\eqref{eq:def-intensity-measure}
with the corresponding Radon-Nikodym derivative one could also prove not 
only convergence of $K^+$-reductions, but also of the 
\emph{modified} $K^+$-reductions, where one takes into account 
the time spent at $\delta$. 
Because all this relies on the unknown uniform integrability assumption,
we do not work out the details here. 

\medskip\noindent
\begin{proof}[Proof of Lemma \ref{le:transition-rates-K-reduction}]
A typical path of a Markov jump process on $\tilde K$ starting in $o$ 
consists of an initial piece running from $o$ to $\delta$ and then, 
independent of it, a concatenation of a sequence of i.i.d.\ pairs,
each consisting 
of an exponential waiting time at $\delta$ and, again independent of it, a 
Markovian loop around $\delta$. The $K^+$-reduction $\omega^K$ 
is indeed constructed in this way:  
\begin{itemize}
\item The initial piece is the $K$-reduction of $\omega_\start$, 
which is independent of $\omega^\leftrightarrow$. According to Lemma 
\ref{le:rates-reduction-path} it is a Markov jump process with 
transition rates given 
by \eqref{eq:transition-rates-short1} and \eqref{eq:transition-rates-short2}.
These rates coincide with the rates claimed in the first two lines of 
the right-hand side of \eqref{eq:transition-rates}.
\item The pairs $((\delta,t_{i_j}-t_{i_{j-1}}), (\hat w_{i_j}|_{\N_0})^K)$,
$j\in\N$, are i.i.d.\ as they are obtained from a decorated Poisson 
process. Being functions of $\omega^\leftrightarrow$, they are independent
of the initial piece. Consider a given $j\in\N$. Since the intensity measure of 
$1_{\hat W^*_K\times\R_+^0}\omega^\leftrightarrow$ is the product measure 
$\pi^*[\hat Q_{K,\beta}]\times dt$, the components $t_{i_j}-t_{i_{j-1}}$ and 
$(\hat w_{i_j}|_{\N_0})^K$ are independent.
\begin{itemize}
\item The waiting time $t_{i_j}-t_{i_{j-1}}$ is exponential with the total mass of 
$\pi^*[\hat Q_{K,\beta}]$ as its parameter, i.e.\ with parameter 
$\sum_{x\in K}\beta_xe_{K,\beta}(x)$, see formula \eqref{eq:total-mass-hat-Q-K} in 
the appendix.
\item The two-sided infinite path $\hat w_{i_j}$ has the law
$\hat Q_{K,\beta}/\hat Q_{K,\beta}(\hat W)$. Hence, by \eqref{eq:def-Q-K},
the law of $\hat w_{i_j}|_{\N_0}$ is given by 
\begin{align}
\frac{1}{\hat Q_{K,\beta}(\hat W)}\sum_{x\in K}\beta_xe^{2u_{o,x}}Q_{x,\beta}^G(A_K)Q_{x,\beta}^G
=\frac{\sum_{x\in K}\beta_xe_{K,\beta}(x)Q_{x,\beta}^G}{\sum_{x\in K}\beta_xe_{K,\beta}(x)}.
\end{align}
Recall the definition of the measure $Q_{x,\beta}^{K^+}$ given in Lemma 
\ref{le:rates-reduction-path}. 
The $K$-reduction of $\hat w_{i_j}|_{\N_0}$, which describes the $j$-th
excursion from $\delta$, therefore has the law 
\begin{align}
\label{eq:sum-Q-x-u-K}
\frac{\sum_{x\in K}\beta_xe_{K,\beta}(x)Q_{x,\beta}^{K^+}}{\sum_{x\in K}\beta_xe_{K,\beta}(x)}.
\end{align}
Conditionally on the starting point $x\in K$, this is just $Q_{x,\beta}^{K^+}$. 
According to Lemma \ref{le:rates-reduction-path}, it describes a Markov jump 
process with rates 
\eqref{eq:transition-rates-short1}--\eqref{eq:transition-rates-short2} 
stopped before being absorbed. Note that these transition rates do not 
depend on the starting point $x$. They coincide with the ones claimed in the 
first two lines of the right-hand side in \eqref{eq:transition-rates}.
Consequently, the law \eqref{eq:sum-Q-x-u-K} 
describes also a Markov jump process with the same transition rates, 
but with a random starting point having the law
$\sum_{x\in K}\beta_xe_{K,\beta}(x)\delta_x/\sum_{x\in K}\beta_xe_{K,\beta}(x)$.
\end{itemize}
Summarizing, jumps away from $\delta$ occur with the total rate 
$\sum_{x\in K}\beta_xe_{K,\beta}(x)$. Any such jump arrives in a given $y\in K$ with 
probability $\beta_ye_{K,\beta}(y)/\sum_{x\in K}\beta_xe_{K,\beta}(x)$. Multiplying these 
two quantities the transition rate from $\delta$ to $y$ is given by
$\beta_ye_{K,\beta}(y)$, as claimed. 
\end{itemize}
\end{proof}

\subsection{Convergence of transition rates}

\begin{theorem}[Infinite-volume limits]
\label{thm:infinite-volume-limit-rates}
For all finite subsets $K\subset V$ and all $x,y\in K$, one has 
$\rho_o$-a.s.
\begin{align}
\lim_{n\to\infty}e_K^n(x)=e_K(x), \quad 
\lim_{n\to\infty}q_K^n(x,y)=q_K(x,y),\quad 
\lim_{n\to\infty}e^{u_{o,x}^{(n)}}=e^{u_{o,x}}.
\end{align}
In particular, the finite-volume transition rates given in 
\eqref{eq:transition-rates-finite-volume} for the modified $K^+$-reduction 
of the Markov jump process in the environment $\beta$ converge $\rho_o$-a.s.\
to the corresponding infinite-volume quantities: 
\begin{align}
\frac12 C_{xy}e^{u^{(n)}_{o,y}-u^{(n)}_{o,x}}+\beta_xq^n_{K,\beta}(x,y)\konv{n\to\infty} 
&\frac12 C_{xy}e^{u_{o,y}-u_{o,x}}+\beta_xq_{K,\beta}(x,y), \\
\beta_xe^{-2u^{(n)}_{o,x}}e_{K,\beta}^n(x)\konv{n\to\infty}
&\beta_xe^{-2u_{o,x}}e_{K,\beta}(x), \\
\beta_ye_{K,\beta}^n(y)\konv{n\to\infty} 
&\beta_ye_{K,\beta}(y). 
\end{align}
\end{theorem}

Note that just as in the last line of the transition law described in 
\eqref{eq:transition-rates-finite-volume}, 
there is no factor $e^{-2u_{o,\delta}^{(n)}}$ on the left-hand side of 
the last equation. 

The proof needs some preliminary lemmas and is given in the 
remainder of this subsection.
Recall the filtration $\F_n=\sigma(\beta_x,x\in V_n)$, $n\in\N$.

\begin{lemma}
\label{le:proba-path-finite-infinite}
Let $n\in\N$, $x,y\in V_n$, and let $\pi=(\pi_0,\pi_1,\ldots,\pi_m)$ be 
a finite path in $G_n$ from $x$ to $y$ with $\pi_k\in V_n$ for all 
$k$. Then, writing $\underline{\pi}$ for the event that the process follows 
the path $\pi$ initially, one has $\rho_o$-a.s.
\begin{align}
\label{eq:proba-path-pi}
Q_{x,\beta}^{G_n}(\underline{\pi}) e^{u_{o,x}^{(n)}}=E_{\rho_o}[ Q_{x,\beta}^G(\underline{\pi})e^{u_{o,x}}|\F_n].
\end{align}
Consequently, if $A$ is the union of countably many such events $\underline{\pi}$, 
one has $\rho_o$-a.s.
\begin{align}
\label{eq:proba-event-A}
Q_{x,\beta}^{G_n}(A) e^{u_{o,x}^{(n)}}=E_{\rho_o}[ Q_{x,\beta}^G(A)e^{u_{o,x}}|\F_n].
\end{align}
\end{lemma}

Note that $\underline\pi$ on the left-hand side in \eqref{eq:proba-path-pi}
is understood as an event in $\hat W_n^\to$, while on the right-hand side in 
\eqref{eq:proba-path-pi} it is understood as an event in $\hat W^\to$.

\medskip\noindent
\begin{proof}[Proof of Lemma \ref{le:proba-path-finite-infinite}]
Using $C^{(n)}_{ab}=C_{ab}$ for all $a,b\in V_n$, we calculate 
\begin{align}
\label{eq:proba-wrt-Q-n-pi}
Q_{x,\beta}^{G_n}(\underline{\pi})=\prod_{k=0}^{m-1}\frac{C^{(n)}_{\pi_k\pi_{k+1}}}{2\beta_{\pi_k}}
e^{u_{o,\pi_{k+1}}^{(n)}-u_{o,\pi_k}^{(n)}}
=e^{u_{o,y}^{(n)}-u_{o,x}^{(n)}}\prod_{k=0}^{m-1}\frac{C_{\pi_k\pi_{k+1}}}{2\beta_{\pi_k}}.
\end{align}
Similarly, we obtain 
\begin{align}
Q_{x,\beta}^G(\underline{\pi})
=e^{u_{o,y}-u_{o,x}}\prod_{k=0}^{m-1}\frac{C_{\pi_k\pi_{k+1}}}{2\beta_{\pi_k}}.
\end{align}
Since all $\beta_{\pi_k}$ are $\F_n$-measurable, the claim 
\eqref{eq:proba-path-pi} follows from 
the equation $e^{u_{o,y}^{(n)}}=E_{\rho_o}[ e^{u_{o,y}}|\F_n]$ given in 
\eqref{eq:relation-eu-finite-vol}.

Taking a countable union $A=\bigcup_{i\in I}\underline{\pi^{(i)}}$ with different $\pi^{(i)}$
and a countable index set
$I$, we may drop all $\pi^{(i)}$ for which there is another $\pi^{(j)}$, 
$j\neq i$, 
being an initial piece of $\pi^{(i)}$. Let $J\subseteq I$ denote the set of 
all remaining indices. Then, $A=\bigcup_{i\in J}\underline{\pi^{(i)}}$ is a 
countable
union of pairwise disjoint events. The claim \eqref{eq:proba-event-A} then 
follows from \eqref{eq:proba-path-pi} and monotone convergence.
\end{proof}

\begin{lemma}
\label{le:proba-path-finite-to-delta}
Let $n\in\N$, $x\in V_n$, and let $\pi=(\pi_0,\pi_1,\ldots,\pi_m)$ be 
a finite path in $G_n$ from $x$ to $\delta$ with $\pi_k\in V_n$ for 
all $0\le k\le m-1$. Let $\Pi_\pi$ denote the set of finite paths in 
the infinite graph $G$ of the form $(\pi_0,\pi_1,\ldots,\pi_{m-1},y)$ 
with $y\not\in V_n$. Let $\underline{\Pi_\pi}$ denote the event that 
the process follows a path in $\Pi_\pi$ initially. 
Then, one has $\rho_o$-a.s.
\begin{align}
Q_{x,\beta}^{G_n}(\underline{\pi}) e^{u_{o,x}^{(n)}}
=E_{\rho_o}[ Q_{x,\beta}^G(\underline{\Pi_\pi})e^{u_{o,x}}|\F_n].
\end{align}
\end{lemma}
\begin{proof}
Similarly to \eqref{eq:proba-wrt-Q-n-pi}, we obtain 
\begin{align}
\label{eq:aux-equ}
Q_{x,\beta}^{G_n}(\underline{\pi})e^{u_{o,x}^{(n)}}= & 
\left(\prod_{k=0}^{m-1}\frac{1}{2\beta_{\pi_k}}\right)
\left(\prod_{k=0}^{m-2} C_{\pi_k\pi_{k+1}}\right)
\cdot C_{\pi_{m-1}\delta}^{(n)}\, e^{u_{o,\delta}^{(n)}}
\end{align}
and for any path $\zeta\in\Pi_\pi$ from $x$ to $y\notin V_n$
\begin{align}
Q_{x,\beta}^G(\underline{\zeta})e^{u_{o,x}}= & 
\left(\prod_{k=0}^{m-1}\frac{1}{2\beta_{\pi_k}}\right)
\left(\prod_{k=0}^{m-2} C_{\pi_k\pi_{k+1}}\right)
\cdot C_{\pi_{m-1}y}\, e^{u_{o,y}}.
\end{align}
Since $\prod_{k=0}^{m-1}\beta_{\pi_k}$ is $\F_n$-measurable, it follows 
\begin{align}
E_{\rho_o}[Q_{x,\beta}^G(\underline{\Pi_\pi})e^{u_{o,x}}|\F_n]
= & \sum_{\zeta\in\Pi_\pi}E_{\rho_o}[Q_{x,\beta}^G(\underline{\zeta})e^{u_{o,x}}|\F_n]
\nonumber\\
=&\left(\prod_{k=0}^{m-1}\frac{1}{2\beta_{\pi_k}}\right)
\left(\prod_{k=0}^{m-2} C_{\pi_k\pi_{k+1}}\right)
\sum_{y\in V\setminus V_n} C_{\pi_{m-1}y}\, 
E_{\rho_o}\left[e^{u_{o,y}}|\F_n\right].
\label{eq:cond-exp-Q-Pi-pi}
\end{align}
Let $y\in V\setminus V_n$. Using the martingale representation
\eqref{eq:relation-eu-finite-vol} and the definition \eqref{eq:def-u} of 
$u_{o,y}^{(n)}$ and $u_{o,\delta}^{(n)}$ together with $u_y^{(n)}=u_\delta^{(n)}=0$, 
cf.\ \eqref{eq:def-psi-n}, yields $\rho_o$-a.s.
\begin{align}
E_{\rho_o}\left[e^{u_{o,y}}|\F_n\right]
=e^{u_{o,y}^{(n)}}=e^{u_y^{(n)}-u_o^{(n)}}=
e^{u_\delta^{(n)}-u_o^{(n)}}=e^{u_{o,\delta}^{(n)}}.
\end{align}
Using the definition of $C^{(n)}_{\pi_{m-1}\delta}$ described above 
\eqref{eq:def-W-n-to}, we obtain $\rho_o$-a.s.
\begin{align}
\sum_{y\in V\setminus V_n} C_{\pi_{m-1}y}\, 
E_{\rho_o}\left[e^{u_{o,y}}|\F_n\right]
=\sum_{y\in V\setminus V_n} C_{\pi_{m-1}y}\, e^{u_{o,\delta}^{(n)}}
=C^{(n)}_{\pi_{m-1}\delta}\, e^{u_{o,\delta}^{(n)}}.
\end{align}
Inserting this in \eqref{eq:cond-exp-Q-Pi-pi} and comparing the result with 
\eqref{eq:aux-equ} the claim follows. 
\end{proof}

The following general lemma on conditional expectations of monotone sequences
is needed in the sequel. 

\begin{lemma}
\label{eq:le-mg-convergence}
On some probability space, let $L^1\ni X_n\ge 0$, $n\in\N$, be a decreasing 
or an increasing sequence with the pointwise limit $\lim_{n\to\infty}X_n=X\in L^1$. 
Let $(\G_n)_{n\in\N}$ be a filtration such that all $X_n$
are measurable with respect to $\sigma(\bigcup_n\G_n)$. Then, one has 
\begin{align}
\lim_{n\to\infty}E[X_n|\G_n]=X\quad\text{ a.s.\ and in }L^1. 
\end{align}
\end{lemma}
\begin{proof}
We have $\| X_n-X\|_1\konv{n\to\infty}0$ by dominated convergence, 
and hence 
\begin{align}
\| E[X_n|\G_n]-E[X|\G_n]\|_1\le \| X_n-X\|_1\konv{n\to\infty}0.
\end{align}
Moreover, 
\begin{align}
\| E[X|\G_n]-X\|_1\konv{n\to\infty}0
\end{align}
by the martingale convergence theorem. Together, it follows 
\begin{align}
\| E[X_n|\G_n]-X\|_1
\konv{n\to\infty}0.
\end{align}
This proves convergence in $L^1$. Finally, $(E[X_n|\G_n])_{n\in\N}$
is a non-negative super- or submartingale, given that $(X_n)_{n\in\N}$
is decreasing or increasing, respectively. Hence it converges a.s.\
as well. 
\end{proof}

\medskip\noindent
\begin{proof}[Proof of Theorem \ref{thm:infinite-volume-limit-rates}]
Fix a finite set $K\subset V$ and $x,y\in K$. 
Recall that $u_{o,x}=\lim_{n\to\infty} u_{o,x}^{(n)}\in\R$ holds 
$\rho_o$-a.s.\ by Lemma \ref{le:lim-u-n-finite}. 
In particular, $\lim_{n\to\infty}e^{u_{o,x}^{(n)}}=e^{u_{o,x}}$
$\rho_o$-a.s.

Given $n\in\N$, the event 
$A=\{1<\widetilde\hit_K=\widetilde\hit_y<\hit_\delta\}\subseteq\hat W^\to_n$ 
of returning to $K$ at $y$ restricted to decorated paths starting at $x$ 
can be written as a countable union of events $\underline\pi$ with 
finite paths $\pi$ from $x$ to $y$ which do not hit $\delta$.   
In particular, equation \eqref{eq:proba-event-A} holds for it. 
This yields 
\begin{align}
q_{K,\beta}^n(x,y) 
= & Q_{x,\beta}^{G_n}(1<\widetilde\hit_K=\widetilde\hit_y<\hit_\delta)\nonumber\\
= & e^{-u_{o,x}^{(n)}}
E_{\rho_o}[ Q_{x,\beta}^G(1<\widetilde\hit_K=\widetilde\hit_y<\hit_{V\setminus V_n})e^{u_{o,x}}|\F_n]. 
\end{align}
Consider the increasing sequence 
$X_n=Q_{x,\beta}^G(1<\widetilde\hit_K=\widetilde\hit_y<\hit_{V\setminus V_n})e^{u_{o,x}}\ge 0$, $n\in\N$. 
Its pointwise limit as $n\to\infty$ is given by 
\begin{align}
X=Q_{x,\beta}^G(1<\widetilde\hit_K=\widetilde\hit_y<\infty)e^{u_{o,x}}=q_{K,\beta}(x,y)e^{u_{o,x}}.
\end{align} 
Clearly, all $u_{o,z}$, $z\in V$, are $\F_\infty$-measurable
and hence the same is true for all $X_n$. Furthermore, 
$E_{\rho_o}[e^{u_{o,x}}]=1<\infty$ by \eqref{eq:relation-eu-finite-vol}.
Hence, $X_n,X\in L^1(\rho_o)$. An application of Lemma 
\ref{eq:le-mg-convergence} yields $\rho_o$-a.s.
\begin{align}
\lim_{n\to\infty}q_{K,\beta}^n(x,y) 
= & e^{-u_{o,x}} X = q_{K,\beta}(x,y).
\end{align}

Similarly, the event $\{\hit_\delta<\widetilde\hit_K\}\subseteq\hat W^\to_n$ 
of hitting $\delta$ 
before returning to $K$ restricted to decorated paths starting at $x$ can 
be written as a countable union of finite paths from $x$ to $\delta$. 
Hence, inserting the definition \eqref{eq:def-e-K-n} of $e_K^n(x)$ and 
applying Lemma \ref{le:proba-path-finite-to-delta} and monotone
convergence, we obtain 
\begin{align}
e_K^n(x) 
=e^{2u_{o,x}^{(n)}}Q_{x,\beta}^{G_n}(\hit_\delta<\widetilde\hit_K)
=e^{u_{o,x}^{(n)}}E_{\rho_o}[ Q_{x,\beta}^G(\hit_{V\setminus V_n}<\widetilde\hit_K)e^{u_{o,x}}|\F_n]. 
\end{align}
We apply Lemma \ref{eq:le-mg-convergence} with the decreasing sequence
\begin{align}
X_n=Q_{x,\beta}^G(\hit_{V\setminus V_n}<\widetilde\hit_K)e^{u_{o,x}}
\konv{n\to\infty}Q_{x,\beta}^G(\widetilde\hit_K=\infty)e^{u_{o,x}}
\end{align}
in $L^1(\rho_o)$. This yields the following $\rho_o$-a.s., using the definition 
\eqref{eq:def-e-K-u-infinite-volume} of $e_k(x)$:
\begin{align}
\lim_{n\to\infty}e_K^n(x) 
=\lim_{n\to\infty}e^{u_{o,x}^{(n)}}\cdot Q_{x,\beta}^G(\widetilde\hit_K=\infty)e^{u_{o,x}}
=e_K(x).
\end{align}
\end{proof}

\subsection{Proof of Theorem \ref{thm:convergence-interlacement}}

The following theorem shows that the finite-dimensional distributions of the 
modified $K^+$-reduction of VRJP on $G_n$ converge weakly as $n\to\infty$
to the finite-dimensional distributions of the modified $K^+$-reduction of the 
random interlacement.

\begin{theorem}{\bf (Convergence of modified $K^+$-reductions)}
\label{thm:convergence-interlacement-modified}
For any finite $K\subset V$ with $o\in K$ and $J\in\N$, one has 
\begin{align}
\cL_{P^n_o}\left(\hat w_\mod^K|_{[0,J]}\right)
\stackrel{\text{w}}{\longrightarrow}
\cL_{\pstrich_o}\left(\omega_\mod^K|_{[0,J]}\right)
\quad\text{ as }n\to\infty. 
\end{align}
\end{theorem}
\begin{proof}
Let $\beta\in B'$ be an infinite-volume environment. 
For $x,y\in\tilde K=K\cup\{\delta\}$, let $r_{x,y}^n=r_{x,y}^n(\beta)$ denote 
the rate given in 
\eqref{eq:transition-rates-finite-volume} for the modified $K^+$-reduction 
of the Markov jump process on $G_n$ and let 
$r_{x,y}^\infty=r_{x,y}^\infty(\beta)$ be the 
corresponding rate in infinite volume given in \eqref{eq:transition-rates}
for the modified $K^+$-reduction of the interlacement. 
Let $r_{x,*}^n=\sum_{y\in\tilde K}r_{x,y}^n$, $n\in\N\cup\{\infty\}$, denote 
the total rate to jump away from $x$. 
By Theorem \ref{thm:infinite-volume-limit-rates} the transition rates converge
$\rho_o$-a.s.: $\lim_{n\to\infty}r_{x,y}^n=r_{x,y}^\infty$ for all $x,y$ in the finite
set $\tilde K$. Let $\Pi_J^{\tilde K}$ denote the set of all paths 
$\pi=(\pi_0,\pi_1,\ldots,\pi_{J+1})\in\tilde K^{[0,J+1]}$ in $\tilde K$ which 
start at $o$. Clearly, $\Pi_J^{\tilde K}$ is a finite set. 
Hence, using Lemma \ref{le:representation-mu-o-n}, 
for any continuous function $f:(\tilde K\times\R_+)^{[0,J]}\to\R$ 
with compact support $\supp f\subseteq(\tilde K\times [M^{-1},M])^{[0,J]}$
for some $M>0$, we have 
\begin{align}
& E_{P^n_o}\left[f(\hat w_\mod^K|_{[0,J]})\right]
=\int_{\R^V_+}\int_{\hat W^\to} f(\hat w_\mod^K|_{[0,J]})\, Q_{o,\beta}^{G_n}(d\hat w) \, d\rho_o
\nonumber\\
= & \int_{\R^V_+}\sum_{\pi\in\Pi_J^{\tilde K}}\int_{\R_+^{[0,J]}} f(\pi|_{[0,J]},\ell)
\prod_{j=0}^J r^n_{\pi_j\pi_{j+1}}e^{-r^n_{\pi_j,*}\ell(j)}
\, d\ell\, d\rho_o\konv{n\to\infty}E_{\pstrich_o}\left[f(\omega_\mod^K|_{[0,J]})\right];
\end{align}
the used dominated convergence is justified by the fact that 
$f$ is compactly supported together with the following bound on the support 
of $f$:
\begin{align}
0\le & r^n_{\pi_j\pi_{j+1}}e^{-r^n_{\pi_j*}\ell(j)}
\le r^n_{\pi_j*}e^{-r^n_{\pi_j*}\ell(j)}\nonumber\\
\le & \ell(j)^{-1}\sup_{x>0}xe^{-x}
\le (e\ell(j))^{-1}\le Me^{-1}. 
\end{align}
In other words, $\cL_{P^n_o}\left(\hat w^K_\mod|_{[0,J]}\right)$
converges vaguely to $\cL_{\pstrich_o}\left(\omega^K_\mod|_{[0,J]}\right)$
as $n\to\infty$. Because vague convergence of a sequence of 
probability measures to a \emph{probability} measure implies 
weak convergence, the claim follows. 
\end{proof}

\medskip\noindent
\begin{proof}[Proof of Theorem \ref{thm:convergence-interlacement}]
Because the original $K^+$-reduction is obtained from its modified 
version just by ignoring the local times at $\delta$, Theorem 
\ref{thm:convergence-interlacement} is an immediate consequence of 
Theorem \ref{thm:convergence-interlacement-modified}. 
\end{proof}

\begin{appendix}
\section{Poisson point process in a fixed environment}
\label{sec:appendix-ppp}

In this appendix, we give a proof of Theorem \ref{thm:intensity-measure}.

\begin{lemma}[Reversibility]
\label{le:reversibility}
For all $\beta\in B'$, $m\in\N$, and all measurable sets 
$A\subseteq(V\times\R_+)^{[0,m]}$, one has
\begin{align}
\sum_{x\in V}\beta_xe^{2u_{o,x}}Q_{x,\beta}^G[(\hat w(k))_{k\in[0,m]}\in A]
=\sum_{x'\in V}\beta_{x'}e^{2u_{o,x'}}Q_{x',\beta}^G[(\hat w(m-k))_{k\in[0,m]}\in A].
\end{align}
In particular, for $x,x'\in V$ and 
$A\subseteq(\{x\}\times\R_+)\times(V\times\R_+)^{[1,m-1]}\times(\{x'\}\times\R_+)$,
one has 
\begin{align}
\label{eq:reversibility-reduced-claim}
\beta_xe^{2u_{o,x}}Q_{x,\beta}^G[(\hat w(k))_{k\in[0,m]}\in A]
=\beta_{x'}e^{2u_{o,x'}}Q_{x',\beta}^G[(\hat w(m-k))_{k\in[0,m]}\in A].
\end{align}
An analogous result holds for $\beta\in B$ and the Markov jump process 
with the law
$Q^{G_n}_{x,\beta}$ on the graph $G_n$ with weights $C^{(n)}$. In particular, 
for any $x,x'\in\tilde V_n$ and any path $\pi=(\pi_0,\pi_1,\ldots,\pi_m)$ 
in $G_n$ from $x$ to $x'$, one has 
\begin{align}
\label{eq:reversibility-single-path}
\beta_{x'} e^{2u_{o,x'}^{(n)}} Q_{x',\beta}^{G_n}(\underline\pi)
=\beta_x e^{2u_{o,x}^{(n)}}
Q_{x,\beta}^{G_n}(\underline{\pi^\leftrightarrow}).
\end{align}
\end{lemma}
Recall that $\underline\pi$ denotes the event that the process initially
follows $\pi$. 

\medskip\noindent\begin{proof}
The argument is the same for the infinite volume version and the finite 
volume version. For this reason we describe it only for infinite volume. 

It suffices to consider measurable sets of the form 
$A=\prod_{k=0}^m(\{x_k\}\times(l_k,\infty))$ with given $x_k\in V$, $l_k\ge 0$
fulfilling $x=x_0$ and ${x'}=x_m$. Then, the claim boils down to 
\eqref{eq:reversibility-reduced-claim} for this special~$A$. 
We express the probability on the left-hand side as follows. Writing 
$B=\prod_{k=0}^m(l_k,\infty)$, it holds 
\begin{align}
&Q_{x,\beta}^G[(\hat w(k))_{k\in[0,m]}\in A]
= \int_B\left(\prod_{k=0}^{m-1}e^{-\beta_{x_k}l(k)}\frac{C_{x_k x_{k+1}}}{2}
e^{u_{o,x_{k+1}}-u_{o,x_k}}\right)
\beta_{x_m}e^{-\beta_{x_m}l(m)}\prod_{k=0}^mdl(k)\nonumber\\
&= e^{-2u_{o,x}}\cdot\beta_{x'}e^{u_{o,{x'}}+u_{o,x}}\int_B 
\left(\prod_{k=0}^{m-1}\frac{C_{x_k x_{k+1}}}{2}\right)
\prod_{k=0}^m e^{-\beta_{x_k}l(k)}\, dl(k).
\label{eq:equation-1-rev}
\end{align}
Indeed, $e^{-\beta_{x_k}l(k)}$ for $k\in[0,m]$ is the probability to remain at
$x_k$ at least a time span of length $l(k)$ after arrival at $x_k$. 
Moreover, $\tfrac{C_{x_k x_{k+1}}}{2}e^{u_{o,x_{k+1}}-u_{o,x_k}}\, dl(k)$ for 
$k\in[0,m-1]$ denotes the probability to jump
from $x_k$ to $x_{k+1}$ in an infinitesimal time span of length $dl(k)$ 
given that the particle is at $x_k$. Similarly, 
$\beta_{x_m}\, dl(m)=\sum_{z\in V} \tfrac{C_{x_m z}}{2}e^{u_{o,z}-u_{o,x_m}}\, dl(m)$
equals the probability to jump from $x_m$ to another site in an infinitesimal 
time span of length $dl(m)$ given that the particle is at $x_m$. 
Using the same argument and the set $B^\leftrightarrow=\prod_{k=0}^m(l_{m-k},\infty)$, 
we obtain 
\begin{align}
& Q_{{x'},\beta}^G[(\hat w(m-k))_{k\in[0,m]}\in A]\nonumber\\
&= e^{-2u_{o,{x'}}}\cdot\beta_xe^{u_{o,{x'}}+u_{o,x}}\int_{B^\leftrightarrow} 
\left(\prod_{k=0}^{m-1}\frac{C_{x_{m-k} x_{m-k-1}}}{2}\right)
\prod_{k=0}^m e^{-\beta_{x_{m-k}}l(k)}\, dl(k)\nonumber\\
&=e^{-2u_{o,{x'}}}\cdot\beta_xe^{u_{o,{x'}}+u_{o,x}}\int_B 
\left(\prod_{k=0}^{m-1}\frac{C_{x_k x_{k+1}}}{2}\right)
\prod_{k=0}^m e^{-\beta_{x_k}l(k)}\, dl(k). 
\end{align}
Comparing this with \eqref{eq:equation-1-rev} finishes the proof of 
the reversibility claim \eqref{eq:reversibility-reduced-claim}, which enables 
us to conclude. 
\end{proof}

Let $\beta\in B'$ and let $K\subseteq V$ be finite with $o\in K$. 
Recall the definitions \eqref{eq:def-Q-K} of $\hat Q_{K,\beta}$,
\eqref{def-H-A} of $\hit_K$, and 
\eqref{eq:def-e-K-u-infinite-volume} of $e_{K,\beta}$. 
The total mass of $\hat Q_{K,\beta}$ equals
\begin{align}
\label{eq:total-mass-hat-Q-K}
\hat Q_{K,\beta}(\hat W)
=\sum_{x\in K}\beta_xe^{2u_{o,x}} Q_{x,\beta}^G[A_K]
=\sum_{x\in K}\beta_xe_{K,\beta}(x). 
\end{align}

\begin{lemma}[Consistency]
\label{le:consistency}
Let $\beta\in B'$ and let $\emptyset\neq K\subseteq K'\subseteq V$ be non-empty finite subsets 
of $V$. For $x\in K$, $x'\in K'$, 
$\ell,\ell'\ge 0$, $m\in\N_0$, $B_1\in \hat\W(\N)$, $B_2\in\hat\W([1,m-1])$,
and $B_3\in \hat\W(\N)$, one has 
\begin{align}
&\hat Q_{K',\beta}[(\hat w(-n))_{n\in\N}\in B_1,\, w(0)=x',\, l(0)\ge \ell',\,
\hit_K(\hat w|_{\N_0})=m,\nonumber\\ 
& \qquad
\hat w|_{[1,m-1]}\in B_2,\, w(m)=x,\, l(m)\ge\ell,
(\hat w(n+m))_{n\in\N}\in B_3]\nonumber\\ 
= & \hat Q_{K,\beta}[(\hat w(-n-m))_{n\in\N_0}\in A_{K'},\,
(\hat w(-n-m))_{n\in\N}\in B_1,\, w(-m)=x',\, l(-m)\ge \ell',\nonumber\\  
& \qquad(\hat w(n-m))_{n\in[1,m-1]}\in B_2,\, w(0)=x,\, l(0)\ge\ell,
\hat w|_\N\in B_3].
\label{eq:claim-consistency}
\end{align}
\end{lemma}
We remark that the case $m\in\{0,1\}$, where $[1,m-1]=\emptyset$, is 
included. 

\medskip\noindent\begin{proof}[Proof of Lemma \ref{le:consistency}]
We consider the case $m=0$ first. If $x\neq x'$, then both sides of 
\eqref{eq:claim-consistency} vanish, being measures of the empty set. 
Assume $x=x'$. The case $B_2=\emptyset$ is trivial. Otherwise, 
using that $\{w(0)=x\}\subseteq\{\hit_K(\hat w|_{\N_0})=0\}$ in the 
first equality and 
$\{w(0)=x,\hat w|_{\N_0}\in A_{K'}\}\subseteq\{\hat w|_{\N_0}\in A_K\}$
in the third equality, we obtain
\begin{align}
&\text{l.h.s.\eqref{eq:claim-consistency}}\nonumber\\
=&\hat Q_{K',\beta}[(\hat w(-n))_{n\in\N}\in B_1,\, w(0)=x,\, 
l(0)\ge \max\{\ell,\ell'\},\, \hat w|_\N\in B_3]\nonumber\\ 
=& \beta_xe^{-\max\{\ell,\ell'\}\beta_x}e^{2u_{o,x}} 
Q^G_{x,\beta}[\hat w|_\N\in B_1,\, \hat w|_{\N_0}\in A_{K'}]\,
Q^G_{x,\beta}[\hat w|_\N\in B_3]\nonumber\\
= & \beta_xe^{-\max\{\ell,\ell'\}\beta_x}e^{2u_{o,x}} 
Q^G_{x,\beta}[\hat w|_\N\in B_1,\, \hat w|_{\N_0}\in A_{K'},\,\hat w|_{\N_0}\in A_K]\,
Q^G_{x,\beta}[\hat w|_\N\in B_3]\nonumber\\
= & \hat Q_{K,\beta}[(\hat w(-n))_{n\in\N_0}\in A_{K'},\,
(\hat w(-n))_{n\in\N}\in B_1,\, w(0)=x,\, l(0)\ge\max\{\ell,\ell'\},\,
\hat w|_\N\in B_3]\nonumber\\
= & \text{r.h.s.\eqref{eq:claim-consistency}}
\end{align}
Next, we treat the remaining case $m\ge 1$. Assume $x'\in K$. Then, 
$\{w(0)=x',\,\hit_K(\hat w|_{\N_0})=m\}=\emptyset$ implies 
$\text{l.h.s.\eqref{eq:claim-consistency}}=0$. Furthermore, the event
$\{w(-m)=x'\}$ is contained in $\{(\hat w(-n))_{n\in\N_0}\notin A_K\}$.
Together with the definition of $\hat Q_{K,\beta}$ this implies 
$\text{r.h.s.\eqref{eq:claim-consistency}}=0$, which proves 
\eqref{eq:claim-consistency} in the case $m\ge 1$, $x'\in K$. 

Finally, assume $x'\in K'\setminus K$. 
Using the definition of $\hat Q_{K',\beta}$, we obtain 
\begin{align}
&\text{l.h.s.\eqref{eq:claim-consistency}}
=\beta_{x'}e^{-\ell'\beta_{x'}}e^{2u_{o,x'}} 
Q^G_{x',\beta}[\hat w|_\N\in B_1,\, \hat w|_{\N_0}\in A_{K'}]\cdot 
\nonumber\\
& Q^G_{x',\beta}[\hit_K(\hat w|_{\N_0})=m,\, 
\hat w|_{[1,m-1]}\in B_2,\, w(m)=x,\, l(m)\ge\ell,
(\hat w(n+m))_{n\in\N}\in B_3].
\label{eq:step1}
\end{align}
Note that given $x'\not\in K$, up to modification on the $Q^G_{x',\beta}$-null set
$\{w(0)\neq x'\}$, 
the event $\{\hit_K(\hat w|_{\N_0})=m\}$ is measurable with respect to 
$\sigma(\hat w|_\N)$ and hence enters only in the last factor on the 
right-hand side in \eqref{eq:step1}. We apply the Markov property at time 
$m$ to the last probability in \eqref{eq:step1}:
\begin{align}
& \text{last factor in \eqref{eq:step1}} \nonumber\\
= & Q^G_{x',\beta}[\hit_K(\hat w|_{\N_0})=m,\, \hat w|_{[1,m-1]}\in B_2,\, w(m)=x] 
e^{-\ell\beta_{x}}Q^G_{x,\beta}[ \hat w|_\N\in B_3].
\label{eq:step17}
\end{align}
An application of the reversibility formula 
\eqref{eq:reversibility-reduced-claim} in the case 
\begin{align}
& \{(\hat w(k))_{k\in[0,m]}\in A\}
=\{w(0)=x',\hit_K(\hat w|_{\N_0})=m,\, \hat w|_{[1,m-1]}\in B_2,\, w(m)=x\}\nonumber\\
=&\{w(0)=x',w(k)\not\in K\text{ for }k\in[1,m-1],\, 
(\hat w(n))_{n\in[1,m-1]}\in B_2,\, w(m)=x\}
\end{align}
yields 
\begin{align}
&Q^G_{x',\beta}[\hit_K(\hat w|_{\N_0})=m,\, \hat w|_{[1,m-1]}\in B_2,\, w(m)=x] \\
=& \frac{\beta_{x}}{\beta_{x'}}
e^{2(u_{o,x}-u_{o,x'})}Q^G_{x,\beta}[w(k)\not\in K\text{ for }k\in[1,m-1],\, 
(\hat w(m-n))_{n\in[1,m-1]}\in B_2,\, w(m)=x']. \nonumber
\end{align}
We insert this in \eqref{eq:step17} and then the result in \eqref{eq:step1}. 
Afterwards, we use the Markov property again. This yields 
\begin{align}
&\text{l.h.s.\eqref{eq:claim-consistency}}
=\beta_{x'}e^{-\ell'\beta_{x'}}e^{2u_{o,x'}} 
Q^G_{x',\beta}[\hat w|_\N\in B_1, \hat w|_{\N_0}\in A_{K'}] \nonumber\\
& \cdot \frac{\beta_{x}}{\beta_{x'}}
e^{2(u_{o,x}-u_{o,x'})} Q^G_{x,\beta}[w(k)\not\in K\text{ for }k\in[1,m-1],
(\hat w(m-n))_{n\in[1,m-1]}\in B_2, w(m)=x']\nonumber\\
&\cdot 
e^{-\ell\beta_{x}}Q^G_{x,\beta}[ \hat w|_\N\in B_3]\nonumber\\
= &\beta_{x}e^{-\ell\beta_{x}}e^{2u_{o,x}}
Q^G_{x,\beta}[w(k)\not\in K\text{ for }k\in[1,m-1],
(\hat w(m-n))_{n\in[1,m-1]}\in B_2,w(m)=x']\nonumber\\
&\cdot e^{-\ell'\beta_{x'}} 
Q^G_{x',\beta}[\hat w|_\N\in B_1,\hat w|_{\N_0}\in A_{K'}] 
Q^G_{x,\beta}[ \hat w|_\N\in B_3]\nonumber\\
= & \beta_{x}e^{-\ell\beta_{x}}e^{2u_{o,x}}
Q^G_{x,\beta}[w(k)\not\in K\text{ for }k\in[1,m-1],\, 
(\hat w(m-n))_{n\in[1,m-1]}\in B_2,\nonumber\\ 
& w(m)=x', \, l(m)\ge\ell', \,
(\hat w(m+n))_{n\in\N}\in B_1,\, (\hat w(m+n))_{n\in\N_0}\in A_{K'}]
Q^G_{x,\beta}[ \hat w|_\N\in B_3]\nonumber\\
= & \text{r.h.s.\eqref{eq:claim-consistency}}.
\end{align}
In the last equality, we have used that $x'\in K'\setminus K$ and 
$x\in K\subseteq K'$ imply 
\begin{align}
& \{ w(-k)\notin K\text{ for all }k\in[1,m-1],\, w(-m)=x',\, 
(\hat w(-n-m))_{n\in\N_0}\in A_{K'},\, w(0)=x\}\nonumber\\
=& \{ (\hat w(-n-m))_{n\in\N_0}\in A_{K'},\, (\hat w(-n))_{n\in\N_0}\in A_K,\, 
w(-m)=x',\, w(0)=x\} .
\end{align}
We conclude that the claim \eqref{eq:claim-consistency} holds in all cases. 
\end{proof}

\medskip
For $K\subseteq V$ finite, defining  
\begin{align}
\hat W_K=\{(w,l)\in\hat W:\, w(0)\in K, w(k)\notin K\text{ for }k<0\}, 
\label{eq:def-W-K-star-old}
\end{align}
the definition of the event $\hat W_K^*$ given in \eqref{eq:def-W-K-star}
can be rewritten as $\hat W_K^*=\pi^*[\hat W_K]$.
Clearly, $\hat Q_{K,\beta}$ is supported on $\hat W_K$. 
For finite subsets $K\subseteq K'$ of $V$, let
\begin{align}
\hat W_{K',K}=\{(w,l)\in\hat W_{K'}:\ w(m)\in K\text{ for some }
m\in\Z\}.
\end{align}
Consider the time shift $\theta_{K',K}:\hat W_K\to\hat W$ 
uniquely characterized by $\operatorname{range}(\theta_{K',K})\subseteq W_{K',K}$
and $\pi^*(\theta_{K',K}(\hat w))=\pi^*(\hat w)$
for $\hat w\in\hat W_K$. Thus the map $\theta_{K',K}$ does nothing but 
a shift of any $\hat w$ such that
its image $\theta_{K',K}(\hat w)$ visits $K'$ for the first time 
at time $0$. Lemma 
\ref{le:consistency} may be rephrased in the following form: 

\begin{lemma}
For all $\beta\in B'$ and all non-empty finite $K\subseteq K'\subset V$, one has 
\begin{align}
\label{eq:claim-theta-Q-hat}
\theta_{K',K}[\hat Q_{K,\beta}]=1_{\hat W_{K',K}}\hat Q_{K',\beta}. 
\end{align}
As a consequence, we obtain 
\begin{align}
\label{eq:consistency-Q-K}
\pi^*[\hat Q_{K,\beta}]=1_{\hat W_K^*}\pi^*[\hat Q_{K',\beta}]\le \pi^*[\hat Q_{K',\beta}].
\end{align}
\end{lemma}
\begin{proof}
Because $\hat Q_{K,\beta}$ is supported on the domain $\hat W_K$ of 
the shift $\theta_{K',K}$, the image measure $\theta_{K',K}[\hat Q_{K,\beta}]$
is indeed well-defined. 
We consider the $\sigma$-fields 
\begin{align}
\hat\W_K:=\{A\in\hat\W: A\subseteq\hat W_K\}, \qquad
\hat\W_{K',K}:=\{A\in\hat\W: A\subseteq\hat W_{K',K}\}.   
\end{align}
With parameters as in Lemma \ref{le:consistency}, the set $\B_{K',K}$ of events 
of the form 
\begin{align}
D=\{ & (\hat w(-n))_{n\in\N_0}\in A_{K'},\, (\hat w(-n))_{n\in\N}\in B_1,\, w(0)=x',\, l(0)\ge \ell',\,
\hit_K(\hat w|_{\N_0})=m,\nonumber\\ 
& \hat w|_{[1,m-1]}\in B_2,\, w(m)=x,\, l(m)\ge\ell,
(\hat w(n+m))_{n\in\N}\in B_3\}
\end{align}
is a generator of $\hat\W_{K',K}$, which is stable under intersections. 
Furthermore, the space $\hat W_{K',K}$ is a countable union of events 
of this form. Therefore it suffices to prove the claim 
\eqref{eq:claim-theta-Q-hat} restricted to $\B_{K',K}$. 

Note that $\text{l.h.s.\eqref{eq:claim-consistency}}=\hat Q_{K',\beta}(D)$; 
the condition $(\hat w(-n))_{n\in\N_0}\in A_{K'}$ comes from the definition 
\eqref{eq:def-Q-K} of the measure $Q_{K',\beta}$. Since 
\begin{align}
\theta_{K',K}^{-1}[D]=
\{ & (\hat w(-n))_{n\in\N_0}\in A_K,\,(\hat w(-n-m))_{n\in\N_0}\in A_{K'},\nonumber\\  
& (\hat w(-n-m))_{n\in\N}\in B_1,\, w(-m)=x',\, l(-m)\ge \ell',\nonumber\\  
& (\hat w(n-m))_{n\in[1,m-1]}\in B_2,\, w(0)=x,\, l(0)\ge\ell,
\hat w|_\N\in B_3\}, 
\end{align}
Lemma \ref{le:consistency} shows that indeed the claim 
\eqref{eq:claim-theta-Q-hat} holds restricted to $\B_{K',K}$. 

Using that the measure $\hat Q_{K',\beta}$ is supported on $\hat W_{K'}$
and $\hat W_{K'}\cap(\pi^*)^{-1}[\hat W_K^*]=\hat W_{K',K}$, we infer 
$1_{\hat W_K^*}\pi^*[\hat Q_{K',\beta}]=\pi^*[1_{\hat W_{K',K}}\hat Q_{K',\beta}]$. 
Using formula \eqref{eq:claim-theta-Q-hat}, the facts that 
$\pi^*\circ\theta_{K',K}=\pi^*$ holds on $\hat W_K$ and that 
the measure $\hat Q_{K,\beta}$ is supported on $\hat W_K$, we conclude 
$\pi^*[1_{\hat W_{K',K}}\hat Q_{K',\beta}]=\pi^*[\hat Q_{K,\beta}]$. 
This proves the equality in the second claim \eqref{eq:consistency-Q-K}. 
The inequality in \eqref{eq:consistency-Q-K} is clear from 
$1_{\hat W_K^*}\le 1$. 
\end{proof}

\medskip\noindent
\begin{proof}[Proof of Theorem \ref{thm:intensity-measure}]
Take any increasing sequence of finite sets $K_n\uparrow V$ as $n\to\infty$. 
From \eqref{eq:consistency-Q-K} we know that $\pi^*[\hat Q_{K,\beta}](A)$
is monotonic in the set argument $K$. We conclude 
\begin{align}
\hat\nu_\beta(A):=\sup_{K\subset V\text{ finite}}\pi^*[\hat Q_{K,\beta}](A)
=\lim_{n\to\infty}\pi^*[\hat Q_{K_n,\beta}](A). 
\end{align}
By monotone convergence, this is $\sigma$-additive in $A$. Hence 
$\hat\nu_\beta$ is a measure. 
The equation \eqref{eq:vu-on-hat-W-K} is an immediate consequence
of \eqref{eq:consistency-Q-K}. Uniqueness follows from the fact
\begin{align}
\hat W^*=\bigcup_{n\in\N} \hat W_{K_n}^*. 
\end{align}
Because all measures $\pi^*[\hat Q_{K,\beta}]$ are finite, the measure
$\hat\nu_\beta$ is $\sigma$-finite. 

The equality in the claim \eqref{eq:non-triviality-nu-beta} is an 
immediate consequence of the restriction property \eqref{eq:vu-on-hat-W-K}.
The finiteness of $\pi^*[\hat Q_{K,\beta}](\hat W_K^*)$ follows from the 
definition \eqref{eq:def-Q-K} of $\hat Q_{K,\beta}$. Finally, given a finite
set $K$ with 
$\emptyset\neq K\subset V$ and $y\in K$, using transience, we take 
$x\in K$ such that with positive probability the Markov jump process 
with law $Q^G_{y,\beta}$ visits $K$ for the last time in $x$. In particular, 
$Q^G_{x,\beta}(A_K)>0$. In view of the definition of $\hat Q_{K,\beta}$, this 
implies the remaining claim $\pi^*[\hat Q_{K,\beta}](\hat W_K^*)>0$. 
\end{proof}
\end{appendix}

\paragraph{Acknowledgment.} 
This work is supported by National Science Foundation of China (NSFC), 
grant No.\ 11771293, and by the Agence Nationale de la Recherche (ANR) 
in France, project MALIN, No.\ ANR-16-CE93-0003.

\end{document}